\documentclass[12pt]{amsart}
\usepackage{comment}
\usepackage[mathscr]{eucal}
\usepackage{amssymb,amsmath}
\usepackage{amsfonts}
\usepackage{a4}
\usepackage{xcolor}
\usepackage{graphicx}
\usepackage{float}

\usepackage{enumerate}

\usepackage[utf8]{inputenc}
\usepackage[T1]{fontenc}

\theoremstyle{plain}
\newtheorem{thm}{Theorem}[section]
\newtheorem{prp}[thm]{Proposition}
\newtheorem{cor}[thm]{Corollary}
\newtheorem{lem}[thm]{Lemma}

\newtheorem{que}[thm]{Question}
\newtheorem{prob}[thm]{Problem}

\theoremstyle{remark}

\newtheorem{fct}[thm]{Fact}
\newtheorem{defi}[thm]{Definition}

\newtheorem{ex}[thm]{Example}

\numberwithin{equation}{section}

\newcommand*{\R}{\mathbb{R}}
\newcommand*{\N}{\mathbb{N}}

\newcommand*{\U}{\mathcal{U}}
\newcommand*{\V}{\mathcal{V}}
\newcommand*{\D}{\mathcal{D}}
\newcommand*{\Dp}{\mathcal{D}'}
\newcommand*{\E}{\mathcal{E}}
\newcommand*{\Ep}{\mathcal{E}'}
\newcommand\mc{\mathcal}

\DeclareMathOperator{\diam}{diam}

\begin{document}

\title{Linear orders on chainable continua}

\author{Witold Marciszewski}
\address{Institute of Mathematics\\
University of Warsaw\\ Banacha 2\newline 02--097 Warszawa\\
Poland\\
ORCID identifier: 0000-0003-3384-5782}
\email{wmarcisz@mimuw.edu.pl}

\author{Julia Ścisłowska}
\address{Doctoral School of Exact and Natural Sciences UW\\
University of Warsaw\\ Banacha 2\\ 02--097 Warszawa\\
Poland}
\email{j.scislowska@uw.edu.pl}

\author{Benjamin Vejnar}
\address{Faculty of Mathematics and Physics\\
Charles University\\ Sokolovská 49/83 \newline 186 75 Praha 8\\
Czech Republic\\
ORCID identifier: 0000-0002-2833-5385} 
\email{vejnar@karlin.mff.cuni.cz}

\thanks{The second author was supported by the grant funded by the National Science Center ,,Homogeneity and genericity of metric structures: groups, dynamical systems, Banach spaces and C* –  algebras” (UMO-2021/03/Y/ST1/00072). The third author was supported by the grant GACR 24-10705S
}


\begin{abstract} 
We define and study certain linear orders on chainable continua. Those orders depend on a sequence of chains obtained from definition of chainability and on a fixed non-principal ultrafilter on the set of natural numbers. An alternative method of defining linear orders on a chainable continuum $X$ uses representation of $X$ as an inverse sequence of arcs and fixed non-principal ultrafilter on $\N$. We compare those two approaches.

We prove that there exist exactly $2$ distinct ultrafilter orders on any arc, exactly $4$ distinct ultrafilter orders on the Warsaw sine curve, and exactly $2^{\mathfrak{c}}$ distinct ultrafilter orders on the Knaster continuum. We study the order type of various chainable continua equipped with an ultrafilter order and prove that a chainable continuum $X$ is Suslinean if and only if for every ultrafilter order $\leq_{\U}^{\D}$ on $X$, the space $(X, \leq_{\U}^{\D})$ is order isomorphic to $([0,1],\leq)$.

We study also descriptive complexity of ultrafilter orders on chainable continua. We prove that the existence of closed ultrafilter order characterizes the arc and we show that for Suslinean chainable continua,  any ultrafilter order is both of type $F_{\sigma}$ and $G_{\delta}$. On the other hand, we prove that there is no analytic and no co-analytic ultrafilter order on the Knaster continuum.
\end{abstract}

\maketitle

\section{Introduction}\label{intro}

Chainable continua are well-studied objects in general topology and related fields, such as dimension theory and the theory of dynamical systems. On the one hand they might be described as inverse limits of sequences of arcs, so they resemble an arc, which is a simple space with some ``neat'' properties. However, on the other hand chainable continua might be very complicated, which makes them an excellent source of interesting examples.

The main goal of this paper is to define and study certain linear orders on chainable continua. 

\subsection*{Convention} Throughout the text, we will use the convention that the set of natural numbers is equal to the set of positive integers, i.e. $\N = \{1,2,3,...\}$.

\subsection{The ultrafilter orders on chainable continua}
We consider linear orders on a chainable continuum $X$ (for definitions see Section \ref{chainable}) which depend on sequence of chains obtained from definition of chainability, for sequence $(\varepsilon_n)_{n\in\N}$ tending to zero, and on fixed ultrafilter $\U$ on the set of natural numbers. Every chain defines a natural linear preorder on $X$ and desired linear order on $X$ is an ultraproduct of those orders, modulo filter $\U$.

In our paper we investigate properties of such families of orders and their dependence on choosen sequence of chains and an ultrafilter. An alternative method of defining linear orders on chainable continuum $X$ uses representation of $X$ as an inverse sequence of arcs, orders on those arcs and their ultraproduct.


 Let us mention that the idea of considering ultrafilter orders on chainable continua is due to Jakub Różycki. According to our knowledge, the study of such orders on chainable continua is a new concept, which has not been studied before.

\subsection{Structure of the paper}
Below we outline the structure of our paper.

In Section \ref{definitions} we present two definitions of ultrafilter orders on a given chainable continuum $X$ -- Definition \ref{seq_of_chains} which refers to chainability of $X$, and Definition \ref{inv_lim} which uses representation of $X$ as the inverse limit of arcs. We compare those two approaches and show that if $X$ is homeomorphic to the inverse limit of arcs $\varprojlim(I_i,f_i)_{i=1}^{\infty}$, then every ultrafilter order on $\varprojlim(I_i,f_i)_{i=1}^{\infty}$ generates an ultrafilter order on $X$.

Then, in Section \ref{examples} we prove that there exist exactly $2$ distinct ultrafilter orders on any arc (i.e. space homeomorphic to $[0,1]$), exactly $4$ distinct ultrafilter orders on the Warsaw sine-curve and exactly $\mathfrak{c}$ distinct ultrafilter orders on a particular chainable continuum consisting of infinitely many arcwise connected components, described in Example \ref{forest_of_sines}. We also present an example (Example \ref{t}) showing that arc components of a given chainable continuum might appear in a different order when we consider distinct ultrafilter orders on $X$. 

Section \ref{Suslinean} is devoted to study order type of ultrafilter orders on Suslinean chainable continua. Main result of this part of our paper states that if $X$ is a Suslinean continuum equipped with any ultrafilter order $\leq_{\U}^{\D}$, then that space $(X,\leq_{\U}^{\D})$ has the order type of an interval. We also obtain a new characterization of Suslinean chainable continua - they are exactly those chainable continua, for which the ultrafilter order topology is ccc.

In Section \ref{knaster} we study ultrafilter orders on the Knaster continuum. We prove that there exist exactly $2^{\mathfrak{c}}$ distinct ultrafilter orders on the Knaster continuum. We also describe topological properties of a Knaster continuum equipped with order topology generated by a certain ultrafilter order.

In Section \ref{descriptive}  we prove a new characterization of an arc - it is the only chainable continuum on which there exists a closed ultrafilter order. We show also that for Suslinean chainable continua any ultrafilter order is both of type $F_{\sigma}$ and $G_{\delta}$. Then we present a proof of a theorem that there is no analytic and no co-analytic (so in particular -- no Borel) ultrafilter order on the Knaster continuum.

Finally, in Section 8  we study the relationships between the endpoints and the absolute endpoints of chainable continua and the minimal or maximal points of chainable continua equipped with ultrafilter orders.


\section{Background on chainable continua}\label{chainable}

Recall that a chainable continuum is a compact, connected and metrizable topological space $X$, which satisfies the following property: For a fixed metric $d$, generating topology of $X$, and for every $\varepsilon>0$, $X$ can be covered by a finite chain of open sets $d_1, d_2, ...,d_n$ such that diameter of $d_i$ is smaller that $\varepsilon$ for each $i$ (we say that sequence of sets $d_1, d_2, ...,d_n$ is a chain if for each $i,j \in\{1,...,n\}$ we have $d_i \cap d_j \neq \varnothing \iff |i-j|\leq 1$). The elements of the chain, i.e., the sets $\{d_i\}_{i=1}^{n}$, will be called links. Equivalently, a chainable continuum is an inverse limit of a sequence of arcs.

We will use the convention from Bing's work \cite{bing} -- we will denote chains with uppercase letters (e.g., $D,E,F...$), and links of chains with lowercase letters (e.g., $\{d_i\}_i, \{e_j\}_j, \{f_k\}_k...$).

If $\{D_n\}_{n=1}^{\infty}$ is a sequence of chains in the space $X$, then we will use the symbol $d_{i,j}$ to denote the $i-th$ link in the $j-th$ chain of the sequence $\{D_n\}_{n=1}^{\infty}$, i.e. the $i$-th link of the chain $D_j = \{d_{1,j},...,d_{i-1,j},d_{i,j},d_{i+1,j},...,d_{k_j,j}\}$.

For a metric space $(X,d)$ and for $\mc{A}=\{A_1,...,A_n\}$ being a family of subsets of $X$, we define $mesh(\mc{A})$ as:
    $$mesh(\mc{A}) = \max\{diam(A_i):A_i\in \mc{A}\}.$$

    If $E = \{e_1,...,e_n\}$ is a chain in continuum $X$ and $mesh(E) <\varepsilon$ for a given $\varepsilon>0$, then we say that $E$ is an $\varepsilon$-chain.  Thus, a metric continuum $X$ is chainable if and only if for every $\varepsilon>0$ there exists an $\varepsilon$-chain covering $X$.   

        The following fact will be useful to work with chainable continua.
\begin{fct}
Let $(X,d)$ be a metric continuum. Then the following conditions are equivalent.
\begin{enumerate}
\item $X$ is a chainable continuum.
\item There is an infinite sequence of chains $D_1,D_2,D_3,...$ such that for every $n \in \N $, the chain $D_n$ covers $X$ and $mesh(D_n) \xrightarrow[]{n \to \infty } 0$. \end{enumerate}

\end{fct}

Examples of chainable continua include: any arc (i.e., any homeomorphic image of a closed interval $[0,1]$), the Warsaw sine curve (described in the Example \ref{sinus}), and the Knaster continuum (described in the Section \ref{knaster}).

It is worth noting, however, that the class of chainable continua is much richer – it can be shown that there exists $\mathfrak{c}$ pairwise non-homeomorphic chainable continua \cite{debski}.

Chainable continua have many interesting topological properties: they are a-triodic, hereditary unicoherent, and have the fixed point property. For a more detailed treatment of chainable continua, see the articles
 \cite{bing-2}, \cite{nadler_2_skl} and the monographs \cite{nadler} and \cite{macias}.

\section{Definitions of ultrafilter orders and some of their basic properties}\label{definitions}

\begin{defi}[Linear order on subsets of $X$]\label{arc_comp_order}
    Let $\{A_i: i\in I\}$ be subsets of a set $X$ and let $\leq$ be a linear order on $X$. For $i,j \in I, i\neq j$ we introduce the notation: $$A_i \leq A_j \iff \forall_{y\in A_i} \forall_{z\in A_j}~ y\leq z. $$
\end{defi}

Below we introduce the key definition of this paper.
\begin{defi}\label{seq_of_chains}
Let $X$ be a chainable continuum and let $\mc{D} = \{D_n\}_{n\in \N}$, (where for $n\in \N$, $D_n= \{d_{s,n}\}_{s=1}^{k_n}$) 
be a sequence of chains covering $X$, such that 
$mesh(D_n) \xrightarrow[]{n \to \infty } 0$. Let $\U$ be a non-principal ultrafilter on $\N$. Then we can compare any two points ${x,y} \in X$ in the sense of the \textbf{ultrafilter order $\leq_{\U}^{\mc{D}} $} on $X$, which we define as follows:
$$ x\leq_{D_n} y \iff 
\exists_{~i\leq j \leq k_n} ~ x\in d_{i,n},~ y\in d_{j,n},  $$
$$ x\leq_{\U}^{\mc{D}} y \iff \{n\in \N: x\leq_{D_n} y\} \in \U.  $$
Note that the order $ \leq_{\U}^{\mc{D}}$ is an ultraproduct (with respect to the ultrafilter $\U$) of the family of relations $\leq_{D_n}, n\in \N$.

We consider also strict inequality:
$$x <_{\U}^{\D} y \iff x\leq_{\U}^{\D}y \wedge x\neq y.$$  
\end{defi}

Below we check that relation $ \leq_{\U}^{\mc{D}}$ is a linear order on $X$. \begin{itemize}
    \item Reflexivity is obvious.
    \item Transitivity: Suppose that $x\leq_{\U}^{\mc{D}} y$ and $ y\leq_{\U}^{\mc{D}} z$. Then $\{n\in \N: x \leq_{D_n} y\} \in \U $ and $\{n\in \N: y \leq_{D_n} z\} \in \U$, hence $\{n\in \N: x \leq_{D_n} y\} \cap \{n\in \N: y \leq_{D_n} z\} \in \U$. Since $\{n\in\N: x \leq_{D_n} z\} \supseteq \{n\in \N: x \leq_{D_n} y\} \cap \{n\in \N: y \leq_{D_n} z\} \in \U$, we have $x\leq_{\U}^{\mc{D}} z$.
    \item Antisymmetry and linearity follows from the following property:
    \item for any pair of distinct points $x,y\in X$, exactly one of the following two conditions holds: either $x\leq_{\U}^{\mc{D}} y$ or $y\leq_{\U}^{\mc{D}} x$.
    
     Let $\varepsilon = d(x,y)/2$ and let $A = \{n\in \N: mesh(D_n)< \varepsilon\}$. Then $A\in \U$, since $A$ is a cofinite set. Clearly, for $n\in A$ and $d_{i,n}, d_{j,n} \in D_n$ such that $x\in d_{i,n}, y\in d_{j,n}$ we have $d_{i,n}\cap d_{j,n} = \emptyset$, in particular, $i\ne j$.
    Therefore the sets $M^n_x = \{i: x\in d_{i,n}\}$ and  $M^n_y = \{j: y\in d_{j,n}\}$
    are disjoint. Since each of the sets $M^n_x, M^n_y$ is either a  singleton or consists of a pair of consecutive numbers, we have that either $M^n_x < M^n_y$ or  $M^n_x > M^n_y$. Hence, for each $n\in A$, exactly one of the following two conditions holds: either $x\leq_{D_n} y$ or $y\leq_{D_n} x$. 
    
    Let 
    \[A_1 = \{n\in A: x\leq_{D_n} y\}\quad \mbox{and}\quad A_2 = \{n\in A: y\leq_{D_n} x\}.\]
    Since $A = A_1\cup A_2$ and $A_1\cap A_2 = \emptyset$, exactly one of the sets $A_1, A_2$ belongs to $\mc{U}$. This obviously implies the required property of the ultrafilter order. 
\end{itemize} 

Note that Definition \ref{seq_of_chains} depends on a fixed sequence of chains $\mc{D} = \{D_n\}_{n\in \N}$ and on a fixed non-principal ultrafilter $\U$ on $\N$. This means that different choices of a sequence of chains covering $X$ or a non-principal ultrafilter on $\N$ can generate different orders on the continuum $X$.

Note that any homeomorphism $h: X\to Y$ between chainable continua, being uniformly continuous,  transfers ultrafilter orders on $X$ onto ultrafilter orders on $Y$.

In \cite[Theorem 12.19]{nadler}, \cite[Theorem 2.4.22]{macias} and \cite[Chapter 1.12]{inverse} one can find a proof of a classical theorem, characterizing chainable continua.
\begin{thm}
    A continuum $X$ is a chainable continuum if and only if it is homeomorphic to an inverse limit of a sequence of arcs, i.e. a space $\varprojlim (X_i,f_i)_{i=1}^{\infty}$ for $X_i = [0,1]$.
\end{thm}

It is natural to study the relations between the family of ultrafilter orders on a given chainable continuum $X$ and the family of ultrafilter orders on the inverse limit $\varprojlim (X_i,f_i)_{i=1}^{\infty}$ which is homeomorphic to the space $X$. For this purpose, we will introduce the following definition.

\begin{defi}\label{inv_lim}
     Let $\varprojlim (I_i,f_i)_{i=1}^{\infty}$ be an inverse limit of compact intervals. Let $\U$ be a non-principal ultrafilter on $\N$. Then we define the
     \textbf{limit-ultrafilter order} $\leq_{\U}^{(I_i,f_i)_{i=1}^{\infty}}$ on the space $\varprojlim (I_i,f_i)_{i=1}^{\infty}$ by:

    \[ x\leq_{\U}^{(I_i,f_i)_{i=1}^{\infty}} y \iff \{i\in \N: x_i\leq y_i\} \in \U,\]
    for $x,y \in \varprojlim (I_i,f_i)_{i=1}^{\infty}$.
\end{defi}


Similarily as before, limit-ultrafilter order on an inverse limit is also a linear order.
In the rest of this subsection we prove that every limit-ultrafilter order is an ultrafilter order.

\begin{defi}
Let $(X,d)$ be a metric space. We say that $f : X \to [0,1]$ is an \textbf{$\varepsilon$-map} if $f$ is a continuous surjection and for every $t\in [0,1]$ $diam(f^{-1}(t))<\varepsilon$.
\end{defi}

We will use an easy-to-prove lemma from book \cite{macias}.
\begin{lem}[Lemma 2.4.20 in \cite{macias}]\label{macias}
Let $(X,d)$ and $(Y,d')$ be compact metric spaces. Let $\varepsilon>0$. If $f:X \to Y$ is an $\varepsilon-$map, then there exists $\delta >0$ such that $diam(f^{-1}(U))< \varepsilon$ for any $U \subseteq Y$ such that $diam(U)< \delta$.
    
\end{lem}

We also need one more observation (see \cite[proof of Theorem 2.13]{nadler}).

\begin{lem}\label{pze}
    Let $\varprojlim(I_i,f_i)_{i=1}^{\infty} $ be an inverse limit of a sequence of compact intervals. For $n \in \N$ let $p_n: \varprojlim(I_i,f_i)_{i=1}^{\infty} \to I_n$ be the $n$-th projection and let $\gamma_n = \sup\{diam(p_n^{-1}(t)):t\in I_n\} $.
    Then $\gamma_n\xrightarrow[]{n \to \infty } 0$.
\end{lem}

\begin{thm}\label{transfer}
Every limit-ultrafilter order on a chainable continuum is an ultrafilter order.
\end{thm}

 \begin{proof} 
   Let $Y = \varprojlim(I_i,f_i)_{i=1}^{\infty}$ be the inverse limit of a sequence of closed intervals $I_i = I$. 

Let $d$ be a metric generating the topology of $Y$. 
Let $p_n \colon \varprojlim(I_i,f_i)_{i=1}^{\infty}\to I_n$ be the projection onto the $n$-th coordinate and for $n\geq 1$ define the numbers $\gamma_n$ as in Lemma $\ref{pze}$. We put  $\varepsilon_n = \gamma_n + \frac{1}{n}$. Then, for every $n$, $p_n$ is a $\varepsilon_n$-map and, by Lemma \ref{pze}, $\lim_{n\to \infty}\varepsilon_n = 0$. Applying Lemma \ref{macias} for $\varepsilon = \varepsilon_n$ and $f=p_n$, we obtain $\delta_n>0$ such that if $U\subseteq I_n$ and $diam(U)<\delta_n$, then $diam(p_n^{-1}(U))<\varepsilon_n$.

Let $E_n = \{e_{1,n},...,e_{k_n,n} \}$ be a chain of open intervals of diameter less than $\delta_n$, enumerated according to the standard order on $[0,1]$, covering $I_n$. Let $D_n = \{p_n^{-1}(e_{i,n})\}_{i=1}^{k_n}$. By our choice of $\delta_n$, we have  $mesh(D_n)  < \varepsilon_n$, hence $mesh(D_n)\xrightarrow[]{n \to \infty } 0$. 

Take any pair $x=(x_i),y=(y_i)$ of distinct points of $Y$ and put $\varepsilon = d(x,y)/2$. Arguing in a similar way as in the proof of the properties of an ultrafilter order $\leq_{\U}^{\mc{D}}$, one can easily verify that if $mesh(D_n)  < \varepsilon$ then
\[x_n \le y_n \Longleftrightarrow x \le_{D_n} y\,.\]
Since  $mesh(D_n)  < \varepsilon$ for cofinitely many $n$, we conclude that
\[x \leq_{\U}^{(I_i,f_i)_{i=1}^{\infty}} y \Longleftrightarrow x \leq_{\U}^{\mc{D}} y\]
for every non-principal ultrafilter $\U$ on $\N$.    
\end{proof}

Thus we have proved that any limit-ultrafilter order on a chainable continuum is an ultrafilter order. However we don't know if the converse holds, see Question \ref{two_approaches}.

\section{Examples of ultrafilter orders on simple chainable continua}\label{examples}

\subsection{Ultrafilter orders on an arc}
The main goal of this subsection is to show that on arc, i.e. on any space homeomorphic to the closed interval $[0,1]$, there are exactly two distinct ultrafilter orders - one of them coincides with the natural order $<$, and the second one is opposite to the natural order $<$.

\begin{thm}\label{arc}
    Let $X$ be a chainable continuum, let $\U$ be a non-principal ultrafilter on $\N$, and let $\mc{D} = \{D_n\}_{n\in\N}$ be any sequence of chains covering $X$, such that $mesh(D_n) \xrightarrow[]{n \to \infty } 0$. Let $P$ be the interval $(0,1), (0,1]$, or $[0,1]$.

Assume that $L\subseteq X$ and that $L=h(P)$, where $h: P \to L$ is a homeomorphism.

Let $\leq_{\U}^{\D}$ be the order on $X$ generated by the ultrafilter $\U$ and the sequence of chains $\D$, restricted to $L$.
Let $\leq_{L}$ be a natural order on $L$, that is, an order such that for $x,y \in P$,
\begin{equation}\label{d1}
x \leq y \iff h(x) \leq_{L} h(y).
\end{equation}
Then the orders $\leq_{L}$ and $\leq_{\U}^{\mc{D}}$ either coincide or are opposite to each other, i.e.:$$ ( \forall_{x,y \in [0,1]}~ h(x) \leq_{L} h(y) \iff h(x) \leq_{\U}^{\mc{D}} h(y)) \vee ( \forall_{x,y \in [0,1]}~h(x) \leq_{L} h(y) \iff h(x) \geq_{\U}^{\mc{D}} h(y)). $$
\end{thm}

In the proof we will use the following lemma.
\begin{lem}\label{arc_help}
    Let $X$ be a chainable continuum and let $x,y,z \in X$. Let $\mc{D} = \{D_n\}_{n\in\N}$ be any sequence of chains covering $X$, such that $mesh(D_n) \xrightarrow[]{n \to \infty } 0$. Suppose that there exists a continuum $M \subseteq X$ such that $x,y \in M$ and $z\notin M$. Then $$\exists_k~ \forall_{n>k} \neg [(x\leq_{D_n} z \leq_{D_n} y) \vee (y\leq_{D_n} z \leq_{D_n} x)]. $$
\end{lem}

\begin{proof}
    Let $\varepsilon = d(z,M)>0.$ We know that since $mesh(D_n) \xrightarrow[]{n \to \infty } 0$, then there exists $k\in \N$ such that for every $n>k~ mesh(D_n)< \frac{\varepsilon}{2}$. Let $n > k$. Take $i$ such that $z\in d_{i,n}$, then $d_{i,n} \cap M \neq \varnothing$. Let $U= \bigcup_{j<i}d_{j,n}$ and let $V= \bigcup_{j>i}d_{j,n}$. Since $D_n$ is a chain, then $U\cap V = \varnothing$. The sets $U$ and $V$ are open and $M \subseteq U \cup V$. From the connectedness of $M$ we know that $M\subseteq U$ or $M \subseteq V$, which implies desired condition.
\end{proof}

We will also use the following lemma, which is easy to prove.
\begin{lem}\label{order_3}
Let $X$ be a set and let $\leq_1$ and $\leq_2$ be linear orders on X. If for every triplet of points $x,y,z \in X$ the following holds:
\begin{equation}
x \leq_1 y \leq_1 z \Longrightarrow (x \leq_2 y \leq_2 z ) \vee (x \geq_2 y \geq_2 z),
\end{equation}
then the following holds:
\begin{equation}
( \forall_{x,y \in X}~ x \leq_{1} y \iff x \leq_{2} y) \vee (\forall_{x,y \in X}~ x \leq_{1} y \iff x \geq_{2} y).
\end{equation}
\end{lem}

We now present the proof of Theorem $\ref{arc}.$
\begin{proof}
Let $x,y,z \in P$ and let $h: P \to L$ be a homeomorphism. Assume that $x \leq y \leq z$, where $\leq$ is a standard order on $P$. Let $a = h(x), b = h(y), c = h(z).$
From condition \ref{d1}, we know that $a \leq_{L} b \leq_{L} c $.

Let $A = h([x,y]), B = h([y,z])$. Then $A,B \subseteq L$ and $A,B$ are homeomorphic to a closed interval (as continuous and nondegenerate images of a closed interval, contained in $L$).

We know that $a,b \in A$ and $c \notin A$, so by Lemma \ref{arc_help} there exists $k_1$ such that for every $n>k_1$ we have $$\neg [(a\leq_{D_n} c \leq_{D_n} b) \vee (b\leq_{D_n} c \leq_{D_n} a)].$$ 

We also know that $b,c \in B$ and $a \notin B$, so by Lemma \ref{arc_help} there exists $k_2$ such that for every $n>k_2$ we have: $$\neg [(b\leq_{D_n} a \leq_{D_n} c) \vee (c\leq_{D_n} a \leq_{D_n} b)].$$

Thus, for $n > \max\{k_1,k_2\}$ we have $a \leq_{D_n} b \leq_{D_n} c$ or $c \leq_{D_n} b \leq_{D_n} a$. We also know that $\{n\in \N:n > \max\{k_1,k_2\} \} \in \U$ (since cofinite sets belong to a non-principal ultrafilter).

Therefore, we have proven that
$$ a \leq_L b \leq_L c \Longrightarrow (a \leq_{\U}^{\mc{D}} b \leq_{\U}^{\mc{D}} c ) \vee (a \geq_{\U}^{\mc{D}} b \geq_{\U}^{\mc{D}} c).$$

By Lemma \ref{order_3}, this completes the proof.

\end{proof}

We also obtain the following corollary.

\begin{cor}\label{non-mixing}
Let $T$ be an arc component of a chainable continuum $X$, $x,y \in T$ and $z\in X\setminus T$. 
Then $z$ is not in the $\leq_{\U}^{\D}$-interval between $x$ and $y$.
Thus, all arc components are $\leq_{\U}^{\D}$-convex.
\end{cor}

\begin{proof}
Since $x$ and $y$ are in the same arc component of $X$, we can connect them with an arc that does not contain $z$. Then we use Lemma \ref{arc_help}.
\end{proof}


\subsection{Ultrafilter orders on continua $S_1,S_2$ and $S_3$}


\begin{ex}\label{sinus}
    Let $S_1$ be the Warsaw sine curve, i.e. the chainable continuum defined as follows: $S_1 = \overline{X}$, where $$X = \{(x, \sin(\frac{1}{x})): x \in (0,\frac{2}{3\pi}]\}
.$$
\centerline {\includegraphics[width=7cm, ]{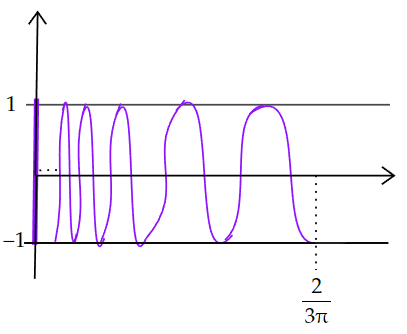} }

\end{ex}

\begin{thm}
    There exist exactly four distinct ultrafilter orders on $S_1$.
\end{thm}

\begin{proof}
    First, we prove that there are at most four distinct ultrafilter orders on $S_1$.

From Theorem \ref{arc}, we know that since the space $\{(x, \sin(\frac{1}{x})): x \in (0,\frac{2}{3\pi}]\}$
is homeomorphic to $(0,\frac{2}{3\pi}]$,
then there are exactly two distinct ultrafilter orders on it. We also know that on the interval $\{0\} \times [-1,1]$ there are exactly two distinct ultrafilter orders. Since the space $S_1$ has only two arc components and on each of them there are exactly two different ultrafilter orders, then on the continuum $S_1$ there are at most four ultrafilter orders (this follows from the fact that any order on $S_1$, when restricted to any arc component, must be an order on that component).

Now we prove that we can find at least four distinct ultrafilter orders on $S_1$. There exist sequences of chains $$\D = \{D_n\}_{n\in \N},~ \Dp=\{D'_n\}_{n\in \N},~ \E=\{E_n\}_{n\in \N}, ~\Ep = \{E'_n\}_{n\in \N},$$ satisfying:
\begin{enumerate}
\item 
for each $n\in\N$, the chains $D_n \in \D , D'_n \in \Dp, E_n \in \E, E'_n \in \Ep$ cover $S_1$,
\item $mesh(D_n) \xrightarrow[]{n \to \infty } 0$, $mesh(D'_n) \xrightarrow[]{n \to \infty } 0$, $mesh(E_n) \xrightarrow[]{n \to \infty } 0$, $mesh(E'_n) \xrightarrow[]{n \to \infty } 0$,
\item for each $n \in \N$, the point $(\frac{2}{3\pi},-1)$ belongs to the first link of the chain $D_n$ and to the first link of the chain $D'_n$,
\item for each $n \in \N$, the point $(0,1)$ belongs to the $m$-th link $d_{m,n}$ of the chain $D_n=\{d_{i,n}\}_{i=1}^{k_n}$, where $$m=\min\{i \leq k_n: d_{i,n}\cap (\{0\} \times [-1,1]) \neq \varnothing\},$$
\item for each $n \in \N$, the point $(0,-1)$ belongs to the $m'$-th link $d'_{m',n}$ of the chain $D'_n=\{d'_{i,n}\}_{i=1}^{k'_n}$, where $$m'=\min\{i \leq k'_n: d'_{i,n}\cap (\{0\} \times [-1,1]) \neq \varnothing\},$$
\item for each $n \in \N$, the chain $E_n \in \E$ is the reversely enumerated chain $D_n$ 
$$\text{ (i.e., } E_n = \{e_{i,n}\}_{i=1}^{k_n} \text{ and for } i \in\{1,...,k_n\}~ e_{i,n} = d_{k_n-i+1, n} \text{)},$$
\item for each $n \in \N$ the chain $E'_n \in \E'$ is the reversely enumerated chain $D'_n$ 
$$\text{ (i.e., } E'_n = \{e'_{i,n}\}_{i=1}^{k'_n} \text{ and for } i \in\{1,...,k'_n\}~ e'_{i,n} = d'_{k'_n-i+1, n} \text{)}.$$
\end{enumerate}

\begin{figure}[]
\centering \includegraphics[width=0.6\textwidth]{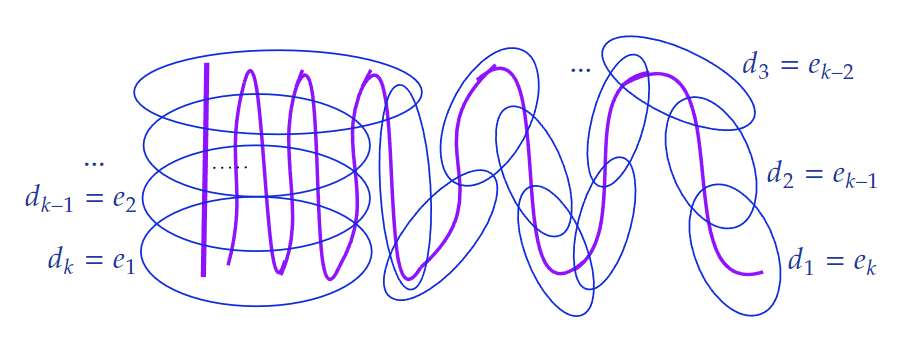}
\caption{Example of chains $D = \{d_i\}_{i=1}^{k} \in \D$ and $E = \{e_i\}_{i=1}^{k}\in \E$.}
\end{figure}

\begin{figure}[]
\centering \includegraphics[width=0.6\textwidth]{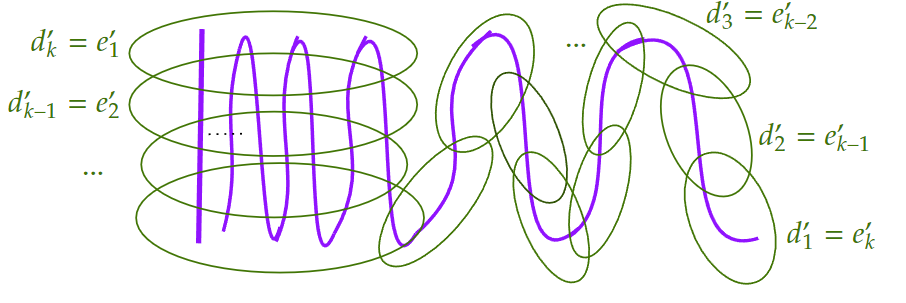}
\caption{Example of chains $D' = \{d'_i\}_{i=1}^{k} \in \D'$ and $E' = \{e'_i\}_{i=1}^{k}\in \E'$.}
\end{figure}

Let us now fix an arbitrary non-principal ultrafilter $\U$ to $\N$. Note that:
$$(0,1)<_{\U}^{\D} (0,-1),~ \Big(\frac{2}{7 \pi},-1\Big)>_{\U}^{\D} \Big(\frac{2}{3 \pi},-1\Big), ~ (0,1)>_{\U}^{\D'} (0,-1),~ \Big(\frac{2}{7 \pi},-1\Big)>_{\U}^{\D'} \Big(\frac{2}{3\pi},-1\Big), $$
$$(0,1)>_{\U}^{\E} (0,-1),~ \Big(\frac{2}{7\pi},-1\Big)<_{\U}^{\E} \Big(\frac{2}{3 \pi},-1\Big), ~ (0,1)<_{\U}^{\E'} (0,-1),~ \Big(\frac{2}{7 \pi},-1\Big)<_{\U}^{\E'} \Big(\frac{2}{3 \pi},-1\Big). $$

This means that for any non-principal ultrafilter $\U$, the orders $\leq_{\U}^{\D}$, $\leq_{\U}^{\D'}$, $\leq_{\U}^{\E}$, $\leq_{\U}^{\Ep}$ are four pairwise distinct ultrafilter orders on $S_1$.
\end{proof}

\begin{ex}
Let $S_2$ be the modified Warsaw sine curve, i.e., the chainable continuum defined as follows: $$S_2 = \overline{X} \cup (\{-1\} \times [-1,1]) \cup ([-1,0] \times \{-1\}),$$ where $$X = \{(x, \sin(\frac{1}{x})): x \in (0,\frac{2}{3\pi}]\}
.$$

 \centerline {\includegraphics[width=8cm, ]{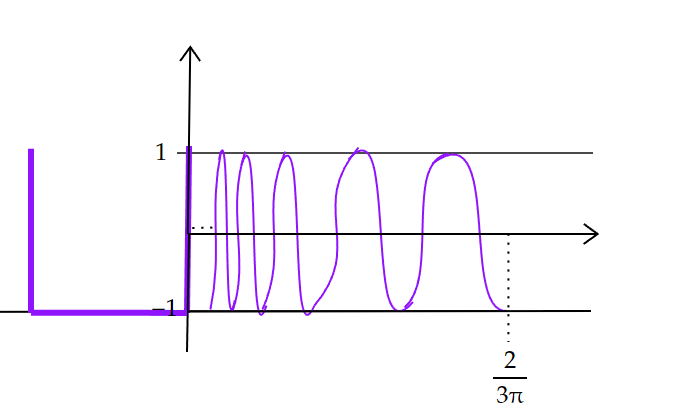} }
\end{ex}

 \begin{thm}
     There exist exactly two distinct ultrafilter orders on $S_2$.
 \end{thm}

 \begin{proof}
     Similarly as in the previous proof, we note that $S_2$ consists of two arc components, and on each of them there are two distinct ultrafilter orders. This means that on the space $S_2$ there are at most four distinct ultrafilter orders. Therefore, for any sequence of chains $\D$ covering $S_2$, such that $mesh(D_n) \xrightarrow[]{n \to \infty } 0$ and for any non-principal ultrafilter $\U$, the ultrafilter order $\leq_{\U}^{\D}$ on $S_2$ must satisfy one of the following conditions:
\begin{enumerate}
\item $(-1,-1)<_{\U}^{\D} (-1,1),~ \big(\frac{2}{7 \pi},-1\big)>_{\U}^{\D} \big(\frac{2}{3 \pi},-1\big)$,
\item $(-1,-1)<_{\U}^{\D} (-1,1),~ \big(\frac{2}{7 \pi},-1\big)<_{\U}^{\D} \big(\frac{2}{3 \pi},-1\big)$, 
\item $(-1,-1)>_{\U}^{\D} (-1,1),~ \big(\frac{2}{7 \pi},-1\big)>_{\U}^{\D} \big(\frac{2}{3 \pi},-1\big)$, 
\item $(-1,-1)>_{\U}^{\D} (-1,1),~ \big(\frac{2}{7 \pi},-1\big)<_{\U}^{\D} \big(\frac{2}{3 \pi},-1\big)$. 

\end{enumerate}
Note that each of the conditions 1.--~4. uniquely determines an ultrafilter order. Since we can reverse enumeration of our chains, we know that on $S_2$ there are at least two distinct ultrafilter orders. Using Theorem \ref{arc_help} one can easily prove that conditions $2.$ and $3.$ cannot hold, which proves the theorem.

 \end{proof}

 \begin{ex}\label{forest_of_sines}
     Consider the chainable continuum $S_3$, which is defined as follows:
$$S_3 = \bigcup_{n\in\N} I_n \cup \bigcup_{n\in\N } A_n \cup \{(0,0)\},$$
where:
\begin{itemize}
\item \normalsize{$I_n $ is a closed interval of the form: $\{\frac{1}{n}\} \times [0, \frac{1}{n}]$,}
\item \normalsize{$A_n = \{\big|x\cdot \sin\big(\frac{1}{(x-\frac{1}{n+1})(\frac{1}{n}-x)}\big)\big| : x\in (\frac{1}{n+1}, \frac{1}{n}) \}.$} 
\item $a_n=\big(\frac{1}{n},0\big) $, $b_n=\big(\frac{1}{n}, \frac{1}{n}\big)$
\end{itemize}
 \centerline {\includegraphics[width=6cm, ]{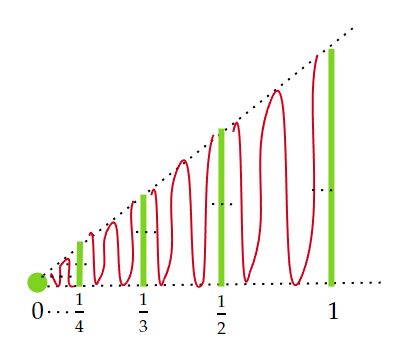} }

 \end{ex}
 

Let $D = \{d_{1},...,d_{k}\}$ be a chain covering the continuum $S_3$ and let $n \in \N$. We say that $D$ \textbf{preserves the orientation of $I_n$} if $a_n \in d_{i}$ and
$b_n \in d_{j} $ for some $i<j\leq k$. Otherwise, we say that $D$ \textbf{reverses the orientation of $I_n$}.


 \begin{thm}\label{thm_S_3}
     There exist exactly $\mathfrak{c}$ distinct ultrafilter orders on $S_3$.
 \end{thm}

 \begin{proof}
     First, we prove that there are at most $\mathfrak{c}$ distinct ultrafilter orders on $S_3$.
The space $S_3$ consists of countably many arc components. Each of them is homeomorphic to $[0,1], (0,1)$ or is a single-point space. This means that on each of the arc components of $S_3$, there are at most two different ultrafilter orders. Therefore, assuming that the arc components of $S_3$ are arranged in a certain way with respect to each other in the ultrafilter order, there are at most $2^{\aleph_0} = \mathfrak{c}$ ultrafilter orders on $S_3$. Since the set of arc components of $S_3$ is countable, we know that there are at most $\mathfrak{c}$ possible ways to arrange these components with respect to each other in an ultrafilter order. Therefore, there are  at most $\mathfrak{c} \cdot \mathfrak{c} = \mathfrak{c} $ possible ultrafilter orders on $S_3$.

Now we prove that there are at least $\mathfrak{c}$ ultrafilter orders on $S_3$. 
First, we claim that for every $s=(s_1,\dots, s_n)\in \{0,1\}^n$, there exists a chain $D_s$  of mesh less than $1/n$ that preserves the orientation of $I_i$ iff $s_i=0$, for every $i\leq n$.
Figures \ref{chain_1}, \ref{chain_2}, and \ref{chain_3} present the choice of $D_s$ for 
$s=0, 01, 011$. We leave the precise definition of $D_s$ to the reader.


\begin{figure}[!]
\centering \includegraphics[width=1\textwidth]{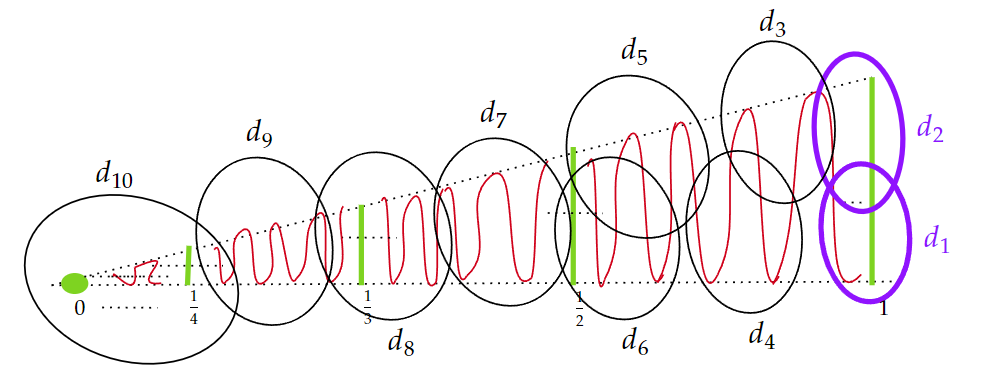}
\caption{Chain $D_0$ is preserving the orientation of $I_1$.}\label{chain_1}
\end{figure}

\begin{figure}[!]
\centering \includegraphics[width=1\textwidth]{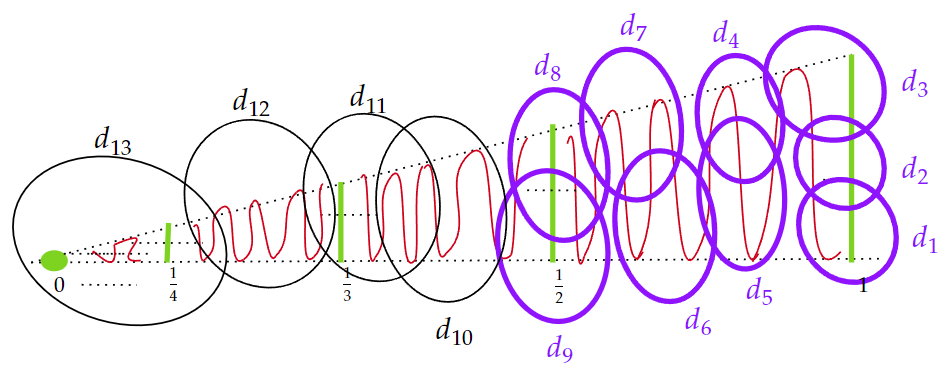}
\caption{Chain $D_{01}$ is preserving the orientation of $I_1$ and reversing the orientation of $I_2$.}\label{chain_2}
\end{figure}

\begin{figure}[!]
\centering \includegraphics[width=1\textwidth]{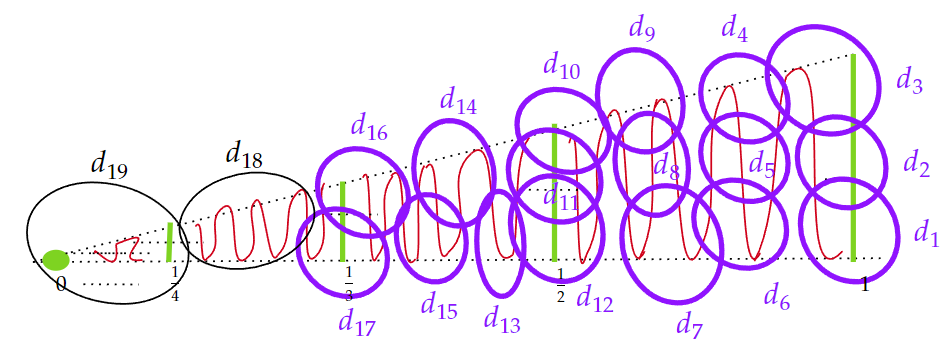}
\caption{Chain $D_{011}$ is preserving the orientation of $I_1$ and reversing the orientation of $I_2, I_3$.}\label{chain_3}
\end{figure}


Let $\U$ be any non-principal ultrafilter on $\N$.
For any $s = (s_n)_{n=1}^{\infty} \in \{0,1\}^{\N}$ we let $\D_s=\{D_{s|n}:n\in\N\}$.
Let $s,t \in \{0,1\}^{\N}$ be a distinct. Then there exists $n\in\N$ such that $s_n\neq t_n$.
Without loss of generality, assume that $s_n=0$ and $t_n=1$.
Then $D_{s|m}$ preserves the orientation of $I_n$ whereas $D_{t|m}$ reverses the orientation of $I_n$, for every $m\geq n$.

Then, for $m\geq n$, we have $a_n \leq_{D_{s|m}} b_n$.
Moreover the set $\{m\in \N: m\geq n\}$ is cofinite, and so it belongs to $\U$.
Thus $a_n \leq_{\U}^{\D_s} b_n$.
Similarly we get that $a_n \geq_{\U}^{\D_t} b_n$.
Thus the ultrafilter orders $\leq_{\U}^{\D_s}$ and $\leq_{\U}^{\D_t}$ are different, since they order differently the pair $(a_n, b_n)$.
Thus, we have proved that there exists $\mathfrak{c}$ pairwise distinct ultrafilter orders on $S_3$.
 \end{proof}

Now we will describe an example which illustrates that the arc components of a given chainable continuum $X$ might be ordered in  many different ways when we consider distinct ultrafilter orders on $X$.

  \begin{ex}\label{t}
There is a chainable continuum $T$ with exactly three arc components: $T_1, T_2, T_3$, described as follows:
 
Recall that $S_1$ is a Warsaw sine curve (described in Example \ref{sinus}), let $T_1$ be the compact arc component in $S_1$, and $T_2 = S_1 \setminus T_1$. 
Let $T_3$ be a ray $(0,1]$ and let $T$ be any metrizable compactification of $T_3$ with remainder $S_1$.
The space $T$ was also described in \cite[Chapter 2, p.34]{thomas}.



Notice that by Corollary \ref{non-mixing} we know that for all $i\neq j$ we have (see Definition \ref{arc_comp_order})
$$T_i \leq_{\U}^{\D} T_j \text{ or } T_j \leq_{\U}^{\D} T_i. $$

One can easily check that there exist sequences of chains $\D = \{D_n\}_{n\in \N} $ and $ \E=\{E_n\}_{n\in \N}$, covering $T$, such that for any non-principal ultrafilter $\U$ on $\N$:
$$ T_3 \leq_{\U}^{\D} T_1 \leq_{\U}^{\D} T_2 \text{ and } T_1 \leq_{\U}^{\E} T_2 \leq_{\U}^{\E} T_3.  $$
Note that the arc component $T_1$ is in the middle in the first ordering, and it is not in the middle in the second ordering of arc components of $T$.
\end{ex}

\section{Order type of ultrafilter orders on Suslinean chainable continua}\label{Suslinean}

\begin{defi}
A continuum is \textbf{Suslinean}
if any collection of its pairwise disjoint nondegenerate subcontinua is countable. 
\end{defi}

Recall that a Borel isomorphism $h$ between metrizable spaces is called a \textbf{Borel isomorphism of the class (1,1)} if both $h$ and $h^{-1}$ are of the first Baire class, i.e., inverse images of open sets under $h$ and $h^{-1}$ are $F_\sigma$-sets.\medskip

In this section we will prove the following result.

\begin{thm}\label{suslin}
    Let $(X,\tau)$ be a non-degenerate Suslinean chainable continuum. Then, for any ultrafilter order $\leq_{\U}^{\D}$ on $X$, the space $X$ with this order  has the order type of an interval, i.e. $$(X, \leq_{\U}^{\D}) \stackrel{\text{izo}}{\simeq} ([0,1],\leq).$$
Moreover, there exists an order isomorphism $h: (X, \leq_{\U}^{\D}) \to ([0,1],\leq)$ which is a Borel isomorphism of the class (1,1) between  $(X,\tau)$ and  $([0,1],\tau_e)$.
\end{thm}

    Let $X$ be a chainable Suslinean continuum. Then $X$ is hereditarily decomposable \cite{minc}, so $X$  satisfies assumptions of the following lemma (which may be found e.g. in \cite[Theorem 1.1]{mohler}).

\begin{lem}\label{NG}
    Let $X$ be a hereditarily decomposable chainable continuum. Then there exists a continuous and monotone surjection $f:X\to [0,1]$ such that if $g:X\to [0,1]$ is any other monotone and continuous surjection, then there is a monotone continuous surjection $m:[0,1] \to [0,1]$ such that $g=m\circ f$. Moreover, for every $t\in[0,1]$ we have $int (f^{-1}(t))=\emptyset$. 
\end{lem}

Following \cite{mohler}, we say that the continua $f^{-1}(t)$ from the above theorem are called\textbf{ tranches} of the continuum $X$.


Note that, by Lemma \ref{arc_help}, monotone maps from a chainable continuum $X$ to the unit interval are also monotone in the sense of the ultrafilter orders (see also Lemma \ref{order_3}):
\begin{lem}\label{porzadek}
            Function $f$ from Lemma \ref{NG} is also monotone with respect to the ultrafilter order $\leq_{\U}^{\D}$, i.e.
for $x,y,z\in X$, for which we have $$x \leq_{\U}^{\D} y \leq_{\U}^{\D} z$$ we must have $$f(x)\leq f(y)\leq f(z) \text{ or }f(x)\geq f(y)\geq f(z).$$
\end{lem}

\begin{proof}
Consider $f:(X,\leq_{\U}^{\D}) \to ([0,1],\leq)$ as in Lemma \ref{NG}. 
   Fix $x,y,z \in X$ with $x \leq_{\U}^{\D} y \leq_{\U}^{\D} z$. Let $J$ be a closed subinterval (possibly degenerate) of $[0,1]$ with endpoints $f(x), f(z)$. Then $f^{-1}(J)$ is a subcontinuum of $X$ containing $x$ and $z$. By our assumption on ordering of the points $x,y,z$ and Lemma \ref{arc_help} we must have that $y\in f^{-1}(J)$. Hence, $f(y)$ belongs to $J$, so it lies between $f(x)$ and $f(z)$.
\end{proof}

For the proof of Theorem \ref{suslin} we will also need the following result of Mohler \cite[Corollary 2.9]{mohler}.

\begin{thm}[Mohler]\label{Mohler_thm}
If $X$ is a hereditarily decomposable chainable continuum, then there is a countable ordinal upper bound on the length of sequences $\{T_{\alpha}\}$ of
nondegenerate subcontinua of $X$ such that 
\begin{itemize}
    \item $T_0$ is a tranche of $X$,
    \item for each $\alpha = \beta+1,~ T_{\alpha}$ is a tranche of $T_{\beta}$,
    \item for limit ordinals $\alpha,~ T_{\alpha} = \bigcap_{\beta<\alpha} T_{\beta}$.
   
\end{itemize}
\end{thm}

\begin{proof}[Proof of Theorem \ref{suslin}]\label{konstrukcja}
    
We will inductively define a family of ordered sets $\{(I_{\alpha},\leq_{\alpha}~): \alpha<\omega_1\}$, a family of maps $\{p_\alpha^\beta: I_{\beta} \to I_\alpha: \alpha<\beta < \omega_1\}$, and   a family of functions $\{f_{\alpha}: X\to I_\alpha: \alpha < \omega_1\}$ satisfying, for each $\alpha<\beta < \omega_1$, the following conditions:
\begin{enumerate}[(a)] 
\item $(I_{\alpha},\leq_{\alpha})$ is order isomorphic to the interval $[0,1]$ with the standard order $\le$;
\item $p_\alpha^\beta: (I_{\beta},\leq_{\beta}) \to (I_{\alpha},\leq_{\alpha})$ is a non-decreasing surjection;
\item $f_{\alpha}$ is a non-decreasing surjection from $(X, \leq_{\U}^{\D})$ to $(I_{\alpha},\leq_{\alpha})$;
\item for each $s\in I_{\alpha}$ the set $f_{\alpha}^{-1}(s)$ is a (possibly degenerate) subcontinuum of $X$;
\item  $f_\alpha: (X,\tau) \to (I_\alpha,\tau_\alpha)$ is of the first Baire class, where $\tau_\alpha$ is the order topology generated by $\leq_{\alpha}$;
\item $f_\alpha = p_\alpha^\beta\circ f_\beta$.
\end{enumerate}

\noindent
\textbf{Initial step:} Let $I_0$ be the closed interval $[0,1]$ with the standard order and let $f_0:X\to I_0$ be a continuous, monotone surjection, obtained from Lemma \ref{NG}. By Lemma $\ref{porzadek}$ function $f_0$ is also order-monotone. Replacing, if necessary, $f_0$ by $1-f_0$, we can assume that $f_0$ is non-decreasing.\medskip

\noindent
\textbf{Successor step:} Suppose that $0 < \alpha<\omega_1$ and $I_\beta, f_\beta, p_\gamma^\beta$ satisfying conditions (a--f) have been already constructed for $\gamma < \beta \le \alpha$.  
Let $$S_{\alpha} = \{s\in I_{\alpha}: |f_{\alpha}^{-1}(s)|>1\}.$$
Since $X$ is Suslinean, the set $S_{\alpha}$ is countable. We define $I_{\alpha+1}$ as the following subset of $I_\alpha \times [0,1]$:
$$I_{\alpha+1} = (I_\alpha \times \{0\}) \cup (S_\alpha \times [0,1])$$
and we declare that $\le_{\alpha+1}$ is a lexicographic order on  $I_{\alpha+1}$. Using  countability of $S_{\alpha}$ one can verify that $(I_{\alpha+1},\leq_{\alpha+1})$ is order isomorphic to the interval $([0,1],\le)$.

The map $p_\alpha^{\alpha+1}$ is the restriction of the projection of $I_\alpha \times [0,1]$ onto the first axis, obviously, it satisfies condition (b). For $\beta < \alpha$,  we put $p^{\alpha+1}_\beta = p_\beta^\alpha \circ p_\alpha^{\alpha+1}$.

For $s \in S_{\alpha}$ let $X_s = f_{\alpha}^{-1}(s)$.
$X_s$ is a Suslinean chainable continuum, so we can apply Lemma \ref{NG}, and  let $f^s_{\alpha+1}: X_s\to [0,1]$ be a monotone continuous surjection given by this lemma. As in the previous step, we may require that  $f^s_{\alpha+1}$ is non-decreasing.  

We define $f_{\alpha+1}: X\to I_{\alpha+1}$ by the formula
$$f_{\alpha+1}(x) = \begin{cases} (f_{\alpha}(x),0)& \mbox{for } x\in X\setminus f_{\alpha}^{-1}(S_\alpha)\,,\\
 (f_{\alpha}(x), f^{f_{\alpha}(x)}_{\alpha+1}(x))& \mbox{for } x\in f_{\alpha}^{-1}(S_\alpha)\,. 
\end{cases}$$
A routine verification shows that conditions (c), (d), and (f) are satisfied. It remains to prove that $f_{\alpha+1}$ satisfies condition (e). Let $a_\gamma = \min (I_{\gamma},\leq_{\gamma})$, and $b_\gamma = \max (I_{\gamma},\leq_{\gamma})$. It is enough to check that for each $u\in I_{\alpha+1}$ with $a_{\alpha+1} <_{\alpha+1} u <_{\alpha+1} b_{\alpha+1}$, the inverse images $f_{\alpha+1}^{-1}((u,b_{\alpha+1}])$ and $f_{\alpha+1}^{-1}([a_{\alpha+1},u))$ are $F_\sigma$-sets in $X$. We will verify this for the first inverse image, the argument for the other one is the same. We have two cases:\\
\textbf{Case 1.} $u = (s,0)$, where $s\notin S_\alpha$. Then 
$$f_{\alpha+1}^{-1}((u,b_{\alpha+1}]) = f_{\alpha}^{-1}((s,b_{\alpha}])$$
which is an $F_\sigma$-set in $X$ by the inductive assumption.\\
\textbf{Case 2.} $u = (s,t)$, where $s\in S_\alpha, t\in [0,1]$. Then 
$$f_{\alpha+1}^{-1}((u,b_{\alpha+1}]) = f_{\alpha}^{-1}((s,b_{\alpha}]) \cup (f^s_{\alpha+1})^{-1}((t,1])$$
which is again an $F_\sigma$-set in $X$ being a union of an $F_\sigma$-set in $X$ and a relatively open subset of a subcontinuum $X_s$.\medskip
    
\noindent
\textbf{Limit step:} Suppose that $\alpha <\omega_1$ is a limit ordinal and $I_\beta, f_\beta, p_\gamma^\beta$ satisfying conditions (a--f) have been already defined for $\gamma < \beta < \alpha$. Since each $(I_{\beta},\leq_{\beta})$ is order isomorphic to the unit interval, $I_{\beta}$ equipped with the order topology is an arc.
We define $I_{\alpha}$ to be the inverse limit of the system of arcs $\{I_{\beta}:\beta < \alpha\}$, together with bonding maps $p^{\beta}_{\gamma} : I_{\beta}\to I_{\gamma}$, $\gamma < \beta < \alpha$. It is known  that inverse limit of a countable sequence of arcs with monotone bonding maps is homeomorphic to the unit interval $[0,1]$ (see e.g.\ \cite[Corollary 2.1.14]{macias} and \cite[Corollary 12.6]{nadler}, and note that $I_{\alpha}$ can be identified with the inverse limit of the sequence (of order type $\omega$)  of arcs $(I_{\beta_n})_{n< \omega}$, where $(\beta_n)_{n< \omega}$ is an increasing sequence of ordinals with supremum equal to $\alpha$), so  the limit space $I_{\alpha}$ is an arc.

For $\beta < \alpha$, we let $p^\alpha_{\beta}: I_{\alpha} \to I_{\beta}$ to be the projection from the inverse limit 
$I_{\alpha}$ onto $\beta-th$ coordinate $I_{\beta}$.

Now, we define an order $\leq_{\alpha}$ on $I_{\alpha}$ in the following way:   
$$ (x_\beta)_{\beta < \alpha} \leq_{\alpha} (y_\beta)_{\beta < \alpha} \iff \forall\, {\beta < \alpha}~ x_{\beta} \leq_{\beta} y_{\beta}$$
for  $(x_\beta)_{\beta < \alpha}, (y_\beta)_{\beta < \alpha} \in I_{\alpha}$.

Since all bonding maps $p^{\beta}_{\gamma}$ are non-decreasing, one can easily verify that 
\begin{equation}\label{limit_order_formula}
(x_\beta)_{\beta < \alpha} <_{\alpha} (y_\beta)_{\beta < \alpha} \iff \exists\, {\beta < \alpha}~ x_{\beta} <_{\beta} y_{\beta}\,.
\end{equation}
From the above it easily follows that the order topology (corresponding to this order) on $I_{\alpha}$ coincides with the topology of the inverse limit, hence $(I_{\alpha},\leq_{\alpha})$ is order isomorphic to the interval $([0,1], \le)$.

We define a function
$$f_{\alpha}:(X,\leq_{\U}^{\D}) \to (I_{\alpha},\leq_{\alpha})$$
by the formula
    $$f_{\alpha}(x)=(f_{\beta}(x))_{\beta < \alpha}\quad \mbox{for } x \in X.$$
    Note that since all functions $\{f_{\beta}:\beta < \alpha\}$ are non-decreasing, the function $f_{\alpha}$ is also non-decreasing.
    
    By the definition of $f_\alpha$, for all $t=(t_{\beta})_{\beta < \alpha}\in I_{\alpha}$, we have
\begin{equation}\label{limit_inverse}
f_{\alpha}^{-1}(t) = \bigcap_{\beta < \alpha}f_{\beta}^{-1}(t_\beta).
\end{equation}    
For all $\gamma < \beta <\alpha$, condition (f) implies that
$$f_\gamma^{-1}(t_{\gamma}) =(p_\gamma^\beta\circ f_\beta)^{-1}(t_{\gamma}) =  f_\beta^{-1}((p_\gamma^\beta)^{-1}(t_{\gamma})).$$
Since $p_\gamma^\beta(t_\beta) = t_\gamma$, we conclude that $f_{\beta}^{-1}(t_\beta) \subseteq f_{\gamma}^{-1}(t_\gamma)$. Hence, by condition (d), the family $\{ f_{\beta}^{-1}(t_\beta):  \beta <\alpha\}$ is a descending (in the sense of inclusion) family of subcontinua of $X$. It follows that $f_{\alpha}^{-1}(t)$ - the intersection of this family is also a subcontinuum of $X$; in particular, this shows that the function $f_\alpha$ is surjective.

Finally, it remains to check that condition (e) also hold for $f_\alpha$. As in the sucessor step, it is enough to verify that, for each $t\in I_{\alpha}$ with $a_{\alpha} <_{\alpha} t <_{\alpha} b_{\alpha}$, the inverse images $f_{\alpha}^{-1}((t,b_{\alpha}])$ and $f_{\alpha}^{-1}([a_{\alpha},t))$ are $F_\sigma$-sets in $X$ (recall that $a_\alpha = \min (I_{\alpha},\leq_{\alpha})$, and $b_\alpha = \max (I_{\alpha},\leq_{\alpha})$). This follows immediately from the inductive assumption,   property \ref{limit_order_formula} of the order $\le_\alpha$, and the definition of $f_\alpha$.
\medskip

By the construction of the functions $f_\alpha$, all their fibers are tranches in $X$. Condition (f) and  property \ref{limit_inverse} of fibers of $f_\alpha$, for limit $\alpha$, implies that, for all $x\in X$, the sequence of tranches  $\{f_\alpha^{-1}(f_\alpha(x))\}$ satisfies the conditions from Theorem \ref{Mohler_thm} (recall that Suslinean continua are hereditarily decomposable). Therefore there exists $\alpha_0<\omega_1$ such that all fibers of $f_{\alpha_0}$ are trivial. This means that the function $f_{\alpha_0}$ is a non-decreasing bijection between $(X, \leq_{\U}^{\D})$ and $(I_{\alpha_0},\leq_{\alpha_0})$, hence it is an order isomorphism. Let $g: (I_{\alpha_0},\leq_{\alpha_0}) \to ([0,1],\leq)$ be an order isomorphism guaranteed by condition (a). We put $h = g\circ f_{\alpha_0}$. Trivially, $h$ is an order isomorphism, and is of the first Baire class, by condition (e) and the obvious fact  that $g$ is a homeomorphism with respect to order topologies.

It remains to verify that $h^{-1}: ([0,1],\tau_e) \to (X,\tau)$ is of the first Baire class, i.e., $h$ maps open sets in $X$ onto $F_\sigma$-sets in $[0,1]$.
   
First, observe that Lemma \ref{arc_help} implies that the image $h(K)$ of any subcontinuum $K$ of $X$ is a subinterval of $[0,1]$, hence an $F_\sigma$-set.

To finish the proof, it is enough to note that each open subset $U$ of $X$ is a countable union of subcontinua of $X$.
Indeed, consider a decomposition of $U$ into constituants (recall that a \textbf{constituant} of a point $x$ in $U$ is a union of all continua containing $x$ and contained in $U$). Each constituant of $U$ contains a nontrivial continuum (cf.\ \cite[\S47.III Thm.4]{kuratowski}), therefore $U$ has countably many constituants, since $X$ is Suslinean. In turn, each constituant of $U$ is a countable union of continua (cf.\ \cite[\S47.VIII Thm.2]{kuratowski}) which gives the desired conclusion.
\end{proof}

Theorem \ref{suslin} is the main tool we need to prove the following characterization, in terms of ultrafilter orders, of chainable continua which are Suslinean.

\begin{thm}
Let $(X,\tau)$ be a chainable continuum equipped with an ultrafilter order $\leq_{\U}^{\D}$. Let $\tau_{\U}^{\D}$ be the order topology generated by order $\leq_{\U}^{\D}$. Then
the following conditions are equivalent:
\begin{enumerate}[(i)]
\item $(X,\tau)$ is Suslinean;
\item $(X, \leq_{\U}^{\D})$ is order isomorphic to $([0,1],\leq)$;
\item $(X,\tau_{\U}^{\D})$ is ccc;
\item the identity map $id: (X, \tau) \to (X,\tau_{\U}^{\D})$ is Borel measurable.
\end{enumerate}
\end{thm}

\begin{proof}
The implication $(i) \Rightarrow (ii)$ is given by Theorem \ref{suslin}. 
The implication $(ii) \Rightarrow (iii)$ is obvious. 
    
Under the assumption of $(i)$, Theorem \ref{suslin} provides us with an order isomorphism $h: (X, \leq_{\U}^{\D}) \to ([0,1],\leq)$ which is a Borel isomorphism of the class (1,1) between  $(X,\tau)$ and  $([0,1],\tau_e)$. Clearly, $h^{-1}: ([0,1],\tau_e) \to  (X,\tau_{\U}^{\D})$ is a homeomorphism. Therefore $id = h^{-1}\circ h: (X, \tau) \to (X,\tau_{\U}^{\D})$ is of the first Baire class, which shows the implication $(i) \Rightarrow (iv)$.

Finally, suppose that $(X,\tau)$ is not Suslinean. Then, there is a well known folklore fact that there exists a collection of size continuum $\mathcal{C} = \{X_\alpha: \alpha < \mathfrak{c}\}$  of pairwise disjoint non-degenerate subcontinua of $X$ (such $\mathcal{C}$ can be constructed  as a copy of the Cantor set in the hyperspace $C(X)$ of subcontinua of $X$ with the help of Kuratowski-Mycielski theorem, see \cite[Theorem 19.1]{kechris}). From each $X_\alpha$ we pick a pair $a_\alpha,b_\alpha$ of distinct points. Without loss of generality we can assume that $a_\alpha <_{\U}^{\D} b_\alpha$ for all $\alpha$. From Lemma \ref{arc_help} we infer that the family $\mathcal{I} = \{(a_\alpha,b_\alpha): \alpha < \mathfrak{c}\}$ has size $\mathfrak{c}$ and consists of  pairwise disjoint nonempty open intervals in the order topology $\tau_{\U}^{\D}$ on $X$. This gives us the implication $(iii) \Rightarrow (i)$. Moreover, for each subfamily $\mathcal{J} \subseteq \mathcal{I}$, its union is an open set in $(X,\tau_{\U}^{\D})$, hence this space has $2^\mathfrak{c} > \mathfrak{c}$ many open sets. Since $(X,\tau)$ has only $\mathfrak{c}$ many Borel sets, the identity map $id: (X, \tau) \to (X,\tau_{\U}^{\D})$ cannot be Borel measurable. This shows the implication $(iv) \Rightarrow (i)$, which completes our proof.
\end{proof}

\section{Ultrafilter orders on the Knaster continuum}\label{knaster}

Let $\mathcal{C} \subseteq [0,1] \times \{0\} \subseteq \R$  be the standard Cantor set.

The following definition is taken from \cite[\S48.V Ex.1]{kuratowski}.
\begin{defi}\label{knasterowskie}
\textbf{The Knaster continuum} is defined as a subspace of $\R^2$, consisting of:
\begin{itemize}
    \item all semi-circles with ordinates $\geq 0$, with center $(\frac{1}{2},0)$ and passing through every point of the Cantor set $\mathcal{C}$,
    \item all semi-circles with ordinates $\leq 0$, which have for $n\geq1$ the center at $(\frac{5}{2\cdot 3^n},0) $ and pass through each point of the Cantor set $\mathcal{C}$, lying in the interval $[\frac{2}{3^n}, \frac{1}{3^{n-1}}]$.
\end{itemize}   
\end{defi}

 In the book \cite{nadler} the Knaster continuum was defined in an alternative way -- it is a space homeomorphic to the inverse limit of the sequence of arcs $\varprojlim (I_i,f_i)_{i=1}^{\infty}$, where, for each $i$, $I_i = [0,1]$ and $f_i = f$, and the map $f: [0,1] \to [0,1] $ is given by:
        \begin{equation}\label{knaster_function}
f(t)=
    \begin{cases}
        2t & \text{for } t\in [0,\frac{1}{2}],\\
        -2t+2 & \text{for } t\in [\frac{1}{2},1].
    \end{cases}
\end{equation}

\begin{thm}
    There exist exactly $2^{\mathfrak{c}}$ distinct ultrafilter orders on the Knaster continuum.
\end{thm}

\begin{proof}
Let us assume that $\U_1$ and $\U_2$ are distinct ultrafilters on $\N$.
Let $\varprojlim (I_i,f_i)_{i=1}^{\infty}$ be the inverse limit representation of the Knaster continuum described above.
We will show that $\leq_{\U_1}^{(I_i,f_i)_{i=1}^{\infty}} \neq \leq_{\U_2}^{(I_i,f_i)_{i=1}^{\infty}}$.

Since $\U_1 \neq \U_2$, there exists an infinite set $A \subseteq \N$ such that $A \in \U_1$ and $\N \setminus A \in \U_2$.
We inductively define sequences $x = (x_0,x_1,...)$ and $y= (y_0,y_1,...)$ such that for every $i$, 
$f(x_i)=x_{i-1}$ and $f(y_i)=y_{i-1}$, ensuring that $x_i > y_i$ if and only if $i\in A$, for $i\geq1$.
\begin{itemize}
\item $i=0:$ Let $x_0=y_0=\frac{1}{2}$.
\item $i=1:$ If $1\in A$, then let $x_1=\frac{3}{4}, y_1=\frac{1}{4}$. If $1\notin A$, then let $x_1=\frac{1}{4}, y_1=\frac{3}{4}$.
\item $i\geq 2:$ 
Let us assume first, that $x_{i-1}<y_{i-1}$.
We know that 
for every $z \in [0,1)$, $f^{-1}(z)=\{z/2, 1-z/2\}$ has two elements.

For every $i$, let $f^{-1}(x_{i-1})=\{x_i^{1},x_i^{2}\}$ and let $f^{-1}(y_{i-1})=\{y_i^{1},y_i^{2}\}$.
We can suppose that $x_i^{1}<x_i^{2}$ and $y_{i}^{1}<y_i^2$. 
Then we have $x_i^1<y_i^1<y_i^2<x_i^2$. Let $y_i$ be any element from the set $f^{-1}(y_{i-1})$.
If $i\in A$ let $x_i=x_i^2$ and if $i\notin A$ let $x_i=x_i^1$.

The case $x_{i-1}>y_{i-1}$, we choose $x_i$ and $y_i$ in a symmetric way. 
\end{itemize}

Therefore, 
we found $x,y\in \varprojlim (I_i,f)_{i=1}^{\infty}$
such that $x_n > y_n \iff n\in A$.
Thus for any ultrafilter $\U$
$$x >_{\U}^{(I_i,f_i)_{i=1}^{\infty}} y 
\iff \{n\in \N: x_n> y_n\} \in \U 
\iff A\in \U.$$

This means that
$x >_{\U_1}^{(I_i,f_i)_{i=1}^{\infty}} y$ and $x <_{\U_2}^{(I_i,f_i)_{i=1}^{\infty}} y$. Therefore, the orders $\leq_{\U_1}^{(I_i,f_i)_{i=1}^{\infty}}$ and $\leq_{\U_2}^{(I_i,f_i)_{i=1}^{\infty}}$ are distinct.


We have thus shown that there are at least as many distinct limit-ultrafilter orders on $\varprojlim (I_i,f_i)_{i=1}^{\infty}$ as there are non-principal ultrafilters on $\N$. By \cite[Theorem 7.6]{jech} we know that there are $2^{\mathfrak{c}}$ such ultrafilters, so this proves that there are $2^{\mathfrak{c}}$ distinct limit-ultrafilter orders on $\varprojlim (I_i,f_i)_{i=1}^{\infty}$. From Theorem \ref{transfer}, we conclude that there exist $2^{\mathfrak{c}}$ different ultrafilter orders on the Knaster continuum. Finally, let us note that there are no more than $2^{\mathfrak{c}}$ different orders, because every ultrafilter order is a relation on the Knaster continuum, and on a set of the cardinality $\mathfrak{c}$ there are exactly $2^{\mathfrak{c}}$ relations.

\end{proof}

\subsection{Order topology generated by a certain ultrafilter order on a Knaster continuum}\label{example}

Let $\U$ be a non-principal ultrafilter on $\N$. Let $\D = \{D_n\}_{n\in \N}$ (where $D_n = \{d_{i,n}\}_{i=1}^{k_n}$) be a sequence of chains covering the Knaster continuum such that for every $n$, $(0,0)\in d_{1,n}$ and $mesh(D_n) \xrightarrow[]{n \to \infty } 0$. It is easy to see that selection of such a sequence of chains is possible.

Consider the Knaster continuum with ultrafilter order $\leq_{\U}^{\D}$. In this part of the paper we will prove that the topological space $(K, \tau_{\U}^{\D})$, i.e. the Knaster continuum equipped with an order topology generated by the order $\leq_{\U}^{\D}$, is a metrizable, non-connected, non-compact and non-separable space.

\begin{defi}
Let $X$ be a continuum and $x\in X$. A \textbf{composant} of a point $x$ is the union of all proper subcontinua of $X$ that contain $x$.
\end{defi}

Let us note that in the Knaster continuum the composants coincide with the arc components \cite[Introduction]{sonja}.

It is also known that the composant of the point $(0,0)$ in the Knaster continuum is a continuous and one-to-one image of a half line, and all the remaining composants of the Knaster continuum (of which there are uncountably many) are continuous and one-to-one images of the open interval \cite[\S48, VI, Examples and remarks]{kuratowski}, \cite[Introduction]{sonja}.

Note that the composant of the Knaster continuum containing the point $(0,0)$ in the space $(K, \tau_{\U}^{\D})$ has the order type of the interval $[0,1)$ (and the point $(0,0)$ is the smallest point in the sense of the order $\leq_{\U}^{\D}$), and all the remaining composants in $(K, \tau_{\U}^{\D})$ have the order type of the open interval. This follows from Theorem \ref{arc} and from the fact that each composant of a Knaster continuum is the sum of an ascending sequence of arcs.

Let us also observe that all composants of a Knaster continuum are open in $(K, \tau_{\U}^{\D})$, which follows from the minimality of the point $(0,0)$, Lemma \ref{arc_help}, and order types of the composants.

 We obtain the following theorem as a corollary from above considerations.

\begin{thm}\label{topo1}
       Let $\U$ be a non-principal ultrafilter on $\N$. Let $\D = \{D_n\}_{n\in \N}$ (where $D_n = \{d_{i,n}\}_{i=1}^{k_n}$) be such a sequence of chains covering the Knaster continuum, that for every $n$, $(0,0)\in d_{1,n}$ and $mesh(D_n)  \xrightarrow[]{n \to \infty } 0$. 

          Then the Knaster continuum with the order topology $\tau_{\U}^{\D}$, generated by an ultrafilter order $\leq_{\U}^{\D}$, is homeomorphic to the disjoint sum of topological spaces $X_i$:
    $$(K, \tau_{\U}^{\D}) \stackrel{\text{homeo}}{\simeq} \bigoplus_{i\in I} X_i, $$
    \small{where $X_0$ is a space homeomorphic to the interval $[0,1)$, corresponding to the arc component of the Knaster continuum containing the point $(0,0)$, and all other $X_i$ are homeomorphic to the open interval $(0,1)$ and correspond to the remaining arc components of the Knaster continuum.}
\end{thm}

\begin{cor}\label{properties}
    The Knaster continuum endowed with the order topology $\tau_{\U}^{\D}$, is a metrizable, non-connected, non-compact and non-separable space.
\end{cor}

\section{Descriptive complexity of ultrafilter orders on chainable continua}\label{descriptive}

Let $X$ be a chainable continuum and let $\leq_{\U}^{\D}$ be an ultrafilter order on $X$. We define the set
$$ M= \{(x,y)\in X^2: x \leq_{\U}^{\D} y\}.$$
The purpose of this part of the paper is to study the descriptive complexity of the set $M$ as a subset of the space $X^2$.

\begin{lem}
  If $X$ is a non-degenerate  chainable continuum and $\leq_{\U}^{\D}$ is an ultrafilter order on $X$, then $M = \{(x,y)\in X^2 : x \leq_{\U}^{\D} y\}$ is not an open subset in $X^2$.
\end{lem}

\begin{proof}
Pick any $x\in X$ and let $U\subseteq X$  be open neighborhood of $x$.
Since $X$ is a non-degenerate continuum, there is a point $y\in U, y\ne x$.
Then both pairs $(x,y), (y,x)$ are in $U\times U$ and exactly one of them is in $M$. Hence, no open neighborhood of $(x,x)$ is a subset of $M$.
\end{proof}

\subsection{Arc}
We showed in Theorem \ref{arc} that if $L$ is an arbitrary arc, then there are exactly two distinct ultrafilter orders on $L$ - one of them coincides with the natural order on the arc $<$, and the other is opposite to the order $<$. We thus obtain the following observation:

\begin{fct}
    Let $L$ be an arc and let $ \leq_{\U}^{\D}$ be an ultrafilter order on $L$. Then the set $M = \{(x,y) \in L^2: x\leq_{\U}^{\D} y\}$ is a closed subset of $L^2$.
\end{fct}
 
It turns out that the existence of an ultrafilter order for which the set $M$ is closed characterizes the arc.

\begin{thm}
    Let $X$ be a chainable continuum and let $\leq_{\U}^{\D}$ be an ultrafilter order on $X$. If the set $M = \{(x,y)\in X^2 : x \leq_{\U}^{\D} y\}$ is closed in $X^2$, then the space $X$ is homeomorphic to the closed interval $[0,1]$.
\end{thm}

\begin{proof}
For clarity, let us denote by $\tau$ the topology on $X$.
Since $M$ is closed in $X^2$, the sets
\begin{align*}
M_a &= \{x\in X: (x,a)\in M\}= \{x\in X: x\leq_{\U}^{\D} a\} = M\cap(X\times \{a\}), \\
M^b &= \{x\in X: (b,x)\in M\}= \{x\in X: b\leq_{\U}^{\D} x\} =  M\cap(\{b\}\times X)
\end{align*}
are closed in $(X,\tau)$, for any $a,b\in X$.
Therefore the sets 
\[ X\setminus M_a = \{x\in X: x>_{\U}^{\D}a\}, \quad
X\setminus M^b = \{x\in X: x<_{\U}^{\D}b\}\] 
are open in $(X,\tau)$.
Thus also the set
\[(X\setminus M_a)\cap (X\setminus M^b) = \{x\in X: a <_{\U}^{\D} x <_{\U}^{\D} b\}\] is open.
This implies that the identity function $id: (X,\tau)\to (X,\tau_{\U}^{\D})$ is continuous. Since $(X,\tau)$ is compact, $id$ is a homeomorphism. Since $(X,\tau_{\U}^{\D})$ is a linearly ordered space, the space $(X,\tau)$, which is homeomorphic to it, is also linearly ordered. Using the fact that every linearly ordered continuum is homeomorphic to a closed interval \cite[6.3.2(b)]{engelking}, the proof is finished.
\end{proof}

\subsection{Suslinean continua}

We showed earlier that for a chainable continuum $X$ and for an ultrafilter order $\leq_{\mathcal{U}}^{\mathcal{D}}$ on $X$, the set $M = \{(x,y)\in X^2 : x \leq_{\mathcal{U}}^{\mathcal{D}} y\}$ is not open in $X^2$ and is usually not closed in $X^2$ (more precisely, an arc is the only chainable continuum $X$ on which there exists an ultrafilter order for which $M$ is closed in $X^2$). However, from Theorem \ref{suslin} we can easily derive  that if $X$ is Suslinean, then the set $M$ is of type $F_{\sigma}$ and $G_{\delta}$ in $X^2$.

\begin{thm}\label{suslin_complexity}
Let $(X,\tau)$ be a Suslinean chainable continuum and let $\leq_{\mathcal{U}}^{\mathcal{D}}$ be an ultrafilter order on $X$. Then the set $$M = \{(x,y)\in X^2 : x \leq_{\mathcal{U}}^{\mathcal{D}} y\}$$ is of both type $F_{\sigma}$ and $G_{\delta}$ in $(X,\tau)^2$.
\end{thm}

\begin{proof}
By Theorem \ref{suslin} there is an order isomorphism $h:(X, \leq_{\U}^{\D}) \to ([0,1],\leq)$ which is a Borel isomorphism of the class (1,1) between  $(X,\tau)$ and  $([0,1],\tau_e)$. Then the map $H = h\times h: (X,\tau)^2 \to ([0,1],\tau_e)^2$ is of the first Baire class. The sets
$$L = \{(s,t) \in [0,1]^2: s <t\} \mbox{ and } G = \{(s,t) \in [0,1]^2: s > t\}$$
are open in $[0,1]^2$, hence, their inverse images $H^{-1}(L),  H^{-1}(G)$ are $F_\sigma$-sets in $X^2$. Therefore the set $M = X^2 \setminus H^{-1}(G)$ is of type $G_\delta$ in $X^2$. Since the diagonal $\Delta_X = \{(x,x): x\in X\}$ is closed in $X^2$, the set $M = \Delta_X \cup H^{-1}(L)$ is also of type $F_\sigma$ in $X^2$.
\end{proof}

The following example shows that the converse of Theorem \ref{suslin_complexity} does not hold. 

\begin{ex}\label{sin_nonsuslinean}
   Let $\mathcal{C} \subseteq [0,1] \subseteq \R$  be the standard Cantor set.

   Let $X$ be a chainable continuum (see Figure \ref{fig:sine}) defined as:

   $$X = (\mc{C}\times [-1,1])\cup \Big\{\Big(t,\sin\Big(\frac{1}{dist(t,\mc{C})}\Big)\Big): t\in [0,1]\setminus \mc{C}\Big\}.$$
   \begin{figure}[h]
   \includegraphics[width=10cm, ]{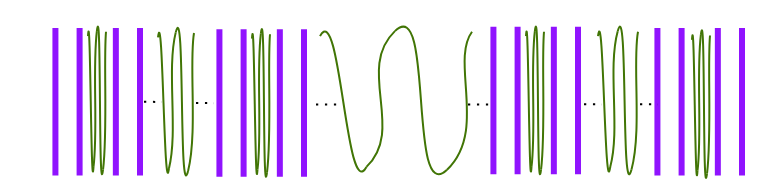} 
   \caption{The continuum $X$.} \label{fig:sine}
    \end{figure}
   Then $X$ is a non-Suslinean chainable continuum.

   Notice that using a similar technique as in the proof of Theorem \ref{thm_S_3} we can obtain a sequence of chains $\mc{D} = \{D_n\}_{n\in \N}$ in $X$, such that 
$mesh(D_n) \xrightarrow[]{n \to \infty } 0$ and, for any non-principal ultrafilter $\mc{U}$, the ultrafilter order $\leq_{\U}^{\D}$ on $X$ is the same as lexicographic order $\leq_{lex}$ on $[0,1]^2$, restricted to $X$ (i.e. $\leq_{\U}^{\D} ~= ~\leq_{lex} \restriction_X $). One can easily check that the lexicographic order $\leq_{lex}$ on $[0,1]^2$ is both of type $F_{\sigma}$ and $G_{\delta}$ (in $[0,1]^4$), hence $\leq_{lex} \restriction_X$ is also both of type $F_{\sigma}$ and $G_{\delta}$ (in $X^2$).   

Notice that if $\tau_{\U}^{\D}$ is the order topology on the space $X$, generated by the ultrafilter order $\leq_{\U}^{\D}$, then the space $(X,\tau_{\U}^{\D}) $ is compact (since each its subset has a supremum) and connected (since it is compact and has no gaps).
Note also that the space $(X,\tau_{\U}^{\D})$ is non-metrizable, since it is compact and not ccc. 
\end{ex}

    \subsection{The Knaster continuum}

The main goal of this section is to present a proof of the following theorem

    \begin{thm}\label{descr_on_knaster}
        For every ultrafilter order $\leq_{\U}^{\D}$ on the Knaster continuum $K$ the set $$M=\{(x,y)\in K^2:x \leq_{\U}^{\D} y\}$$ is a non-analytic and non-co-analytic subset of $K^2$. In particular, for every ultrafilter order $\leq_{\U}^{\D}$ on $K$, the set $M$ is a non-Borel subset of $K^2$.
    \end{thm}

 Let $\leq_{\U}^{\D}$ be any ultrafilter order on $K$. Let $C = K\cap l$, where $l$ is a line described as $x=\frac{1}{2}$. Let $l'$ be a line described as $x=\frac{9}{20}$ and let $C'=K\cap l'$. Notice that $C$ and $C'$ are homeomorphic to the Cantor set. For every $x\in C$ let $H_x$ be the unique semicircle from definition of $K$ (see Definition \ref{knasterowskie}), such that $x\in H_x$ and let $x'\in C'$ be such the point in $H_x$, which lies on line $l'$.
    
    We have a bijective correspondence between points in $ C$ and sequences in $\{0,1\}^{\N}$, described as follows:

    For $y\in \{0,1\}^{\N}$ let $$p(y) = \sum_{n=0}^{\infty} \frac{2y_n}{3^{n+2}}.$$
    Notice that for every binary sequence $y$ there exists exactly one $x \in C$ such that the point $(p(y),0)$ is in $H_{x}$, and $y$ is uniquely determined by $x$. Therefore we can identify points $y \in \{0,1\}^{\N}$ and $x\in C$. This correspondence is a homeomorphism between $C$ and $\{0,1\}^{\N}$.


    

    From now we will be referring to points in $C$ as to infinite binary sequences, using the above correspondence.
    We consider the following partition of $C$ into two sets. Let 
    \[A=\{x\in C: x' <_{\U}^{\D} x \},\quad B = \{x\in C:  x' >_{\U}^{\D} x \}.\]
Then $C=A\cup B$ and $A\cap B=\varnothing$.
For $n\in \N\cup \{0\}$ we define functions $s_n:C\to C$ by 
   \[s_n(x) = s_n((x_0,x_1,x_2,...,x_n,x_{n+1},...)) = (x_0,...,x_{n-1},1-x_n, 1-x_{n+1},... ),\]
for $x\in C$.
We define sets $A_n\subseteq C$ by:
    \begin{itemize}
        \item  $A_0 = C$.
        \item  $A_n = \{x\in C: x_k=0 \text{ for all } k\leq n-2 \text{ and } x_{n-1}=1\}$, for $n>0$. 
    \end{itemize}

    Let $B_{s} = \{x\in C: x\restriction n = s\}$, for $s\in \{0,1\}^n$.

    \begin{defi}
        Let $D\subseteq C$. We say that \textbf{$s_n$ changes orientation on $D$} if $s_n(x) \in B$ when $x\in A\cap D$ and $s_n(x)\in A$ when $x\in B\cap D$.
    \end{defi}

    \begin{lem}\label{czesciowo_dobrze}
     For all $n \in \N\cup\{0\}$, the function $s_n$ changes orientation on $A_n$.
\end{lem}

\begin{proof}
     For $x\in A_n$ let $I_x$ be the unique semicircle with the center at $(\frac{5}{2\cdot 3^{n+1}},0)$ contained in $K$, such that $I_x \cap H_x \neq \varnothing$. Let $J_x$ be the unique semicircle with the center at $(\frac{1}{2},0)$, contained in $K$ and such that $I_x\cap J_x \neq \varnothing$ and $J_x\cap H_x = \varnothing$.

     \begin{figure}[h]
  \centering  \includegraphics[width=0.5\textwidth]{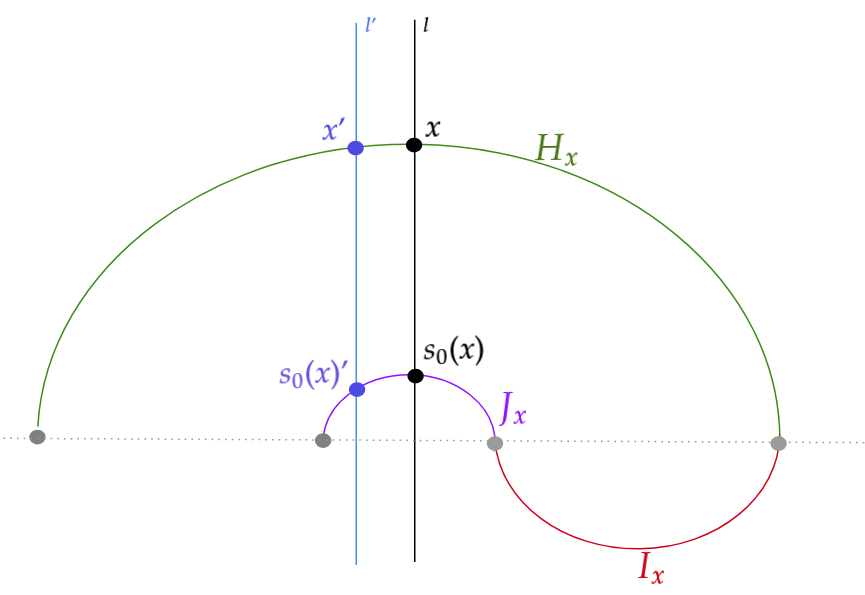}
  \caption{Function $s_0$ on $C$ maps $A$ to $B$ and $B$ to $A$ }\label{fig1}
  \end{figure}

   We consider function $s_n$ on the set $A_n$. Notice that for every $x \in A_n$ there is an arc connecting points $x$ and $s_n(x)$, contained in $H_x\cup I_x\cup J_x$. Similarily, we can connect points $x'$ and $s_n(x)'$ by an arc contained in $H_x\cup I_x\cup J_x$.  

   It is depicted in the Figure \ref{fig1} for $n=0$ that for $x,x' \in H_x$ we have $s_0(x),s_0(x)'\in J_x$ and
   \begin{equation}\label{change_2}
   x <_{\U}^{\D} x' \iff s_0(x) >_{\U}^{\D} s_0(x)'.    
   \end{equation}
   We obtain condition \ref{change_2} with the use of Theorem \ref{arc}.
    A similar argument works analogously also for $n\geq 1$ and function $s_n\restriction_{A_n}$, and shows that for all $n$ function $s_n$ changes orientation on $A_n$. 
\end{proof}

\begin{lem}\label{composing_and_composing}
    For all $ n\in \N\cup\{0\}$ and for all $s\in \{0,1\}^n$, the function $s_n \restriction_{B_s}$ is a composition of an odd number of restrictions of functions $s_{i_k}$, where for all $i_k$ function $s_{i_k}$ is either $s_l$ for some $l<n$ or $s_n \restriction A_n$. 
\end{lem}

\begin{proof}
    Let $s\in \{0,1\}^n$. Then there exist $m_n\in \N$ and functions \\$\{s_{k_i}: 0\leq i\leq m_n \}$ satisfying $k_0<k_1<k_2<...< n$ such that $s_{k_0} \circ...\circ s_{m_n}(B_s) \subseteq A_n$. Apply $s_n\restriction_{A_n}$ to the set $s_{k_0} \circ...\circ s_{m_n}(B_s)$ and notice that $$s_n\restriction_{B_s} = s_{k_0} \circ...\circ s_{m_n}\circ (s_n\restriction A_n)\circ s_{m_n} \circ ... \circ s_{k_0}.$$

\end{proof}


    \begin{lem}\label{change}
        For all $n\in \N\cup\{0\}$, the function $s_n$ changes orientation on $C$.  
    \end{lem}

    \begin{proof}

We will prove this lemma by induction on $n$. For $n=0$ the thesis follows from Lemma \ref{czesciowo_dobrze}. Let $n\in \N$ and assume that for all $k<n$ function $s_k$ changes orientation on $C$. By Lemma \ref{czesciowo_dobrze} we know that $s_n\restriction_{A_n}$ also changes orientation. Fix any $s\in \{0,1\}^n$. By Lemma \ref{composing_and_composing} we know that function $s_n \restriction_{B_s}$ is a composition of odd number of functions $s_{i_k}$, where for all $i_k$ function $s_{i_k}$ is either $s_l$ for some $l<n$ or $s_n \restriction A_n$. We know that all of the functions $s_{i_k}$ change orientation, so $s_n \restriction_{B_s}$, which is a composition of odd number of those functions, also changes orientation. Since the choice of $s\in \{0,1\}^n$ was arbitrary, we conclude that $s_n$ changes orientation on $C$.  

    \end{proof}

    \begin{lem}\label{rozdmuchany_zbior}
    For every open and nonempty $U\subseteq C$ and for every $x\in C$ there exists an even number $k_1$ such that
    $$ x\in s_{i_1}\circ...\circ s_{i_{k_1}}(U) \text{ for some } s_{i_1},...,s_{i_{k_1}},$$
    and an odd number $k_2$ such that  $$ x\in s_{j_1}\circ...\circ s_{j_{k_2}}(U) \text{ for some } s_{j_1},...,s_{j_{k_2}}.$$
\end{lem}

\begin{proof}

  Since $U$ is open and nonempty, there exists $n$ and a binary sequence of length $n$, $s\in \{0,1\}^n$, such that $B_s\subseteq U$. 
  Notice that there are finitely many indices $l_1,...,l_m$ such that $s_{l_1}\circ...\circ s_{l_m}(B_s) = B_{x\restriction n} \ni x$. We know that there must happen exactly one of the following cases:
    \begin{itemize}
        \item $s_{l_1}\circ...\circ s_{l_m} \circ s_{n+1} (B_{s\hat\ {0} }) = B_{x\restriction {(n+1})}$ or 
        \item $s_{l_1}\circ...\circ s_{l_m} \circ s_{n+1} (B_{s\hat\ {1} }) = B_{x\restriction {(n+1})}$ 
    \end{itemize}
    If $m$ is even then let $i_1,...,i_{k_1} =l_1,..., l_m$ and $j_1,...,j_{k_2}= l_1,..., l_m,n+1 $.
If $m$ is odd then let $i_1,...,i_{k_1} =  l_1,..., l_m,n+1$ and $j_1,...,j_{k_2} = l_1,..., l_m$.
\end{proof}

\begin{lem}\label{głowne}
    The set $A \subseteq C$ does not have the property of Baire.
\end{lem}

\begin{proof}
    Suppose that $A$ has the Baire property. Then $B$ also has the Baire property. It follows that $A$ is non-meager or $B$ is non-meager. Without loss of generality $A$ is non-meager. This implies that $A$ is a comeager in some open and nonempty set $U$ (i.e. $A = U\triangle M$, where $U$ is open (in $C$), nonempty, and $M$ is meager). From Lemma \ref{rozdmuchany_zbior} we know that
    \begin{itemize}
        \item set $C$ may be covered by finitely many sets of the form $s_{i_1}\circ... \circ s_{i_{k_1}}(U)$, for some even $k_1$ and some $i_1,...i_{k_1}$,   
        \item set $C$ may be covered by finitely many sets of the form $s_{j_1}\circ... \circ s_{j_{k_2}}(U)$, for some odd $k_2$ and some $j_1,...j_{k_2}$.   
    \end{itemize}
    By the fact that functions $s_i$ are homeomorphisms of $C$, and by Lemma \ref{change}, the following implications hold:
    \begin{itemize}
        \item $A$ is a comeager in $U \implies A$ is a comeager in $ s_{i_1}\circ... \circ s_{i_{k_1}}(U)$ for any $i_1,...,i_{k_1} \implies A$ is a comeager in $C$,
        \item $A$ is a comeager in $U \implies B$ is a comeager in $ s_{j_1}\circ... \circ s_{j_{k_2}}(U)$ for any $j_1,...,j_{k_2} \implies B$ is a comeager in $C$.
    \end{itemize}
    Hence disjoint sets $A$ and $B$ are both comeager in $C$ -- this is a contradiction.
\end{proof}

Now we are ready to present the proof of our main theorem of this subsection.
\begin{proof}[Proof of Theorem \ref{descr_on_knaster}]
    Let $g:C\to C'$ be a function which to each $x\in C$ assigns unique point of $H_x \cap l'$ (in other words, $g(x)=x'$ for each $x\in C$).
    Let $Gr(g)$ be the graph of function $g$. This means that $$Gr(g)=\{(x,g(x)):x\in C\} = \{(x,x'):x\in C\}\subseteq C\times C' \subseteq K\times K.$$
    We know that $Gr(g)$ is a closed subset of $K\times K$ - in fact, it is even homeomorphic to the Cantor set.

    Suppose, towards contradiction, that the set $M=\{(x,y)\in K^2: x\leq_{\U}^{\D}y\}$ is an analytic (co-analytic) set.

    Then $M\cap Gr(g)$ is also an analytic (co-analytic) set. Notice that $$B= \pi_1(M\cap Gr(g)),$$
    where $\pi_1$ is a projection onto the first coordinate. This projection restricted to the graph of $g$ is a homeomorphism, so from the fact that $M\cap Gr(g)$ is analytic (co-analytic) we obtain that $B = \pi_1(M\cap Gr(g))$ is also analytic (co-analytic), hence it has the property of Baire. A contradiction with Lemma \ref{głowne}.
\end{proof}

\section{Endpoints and absolute endpoints of chainable continua}

Let $X$ be a chainable continuum.  Recall that a point $x \in X$ is an \textbf{endpoint} of $X$ if for every $\varepsilon>0$, there is an $\varepsilon$-chain covering $X$ such that only the first link of the chain contains $x$. The following fact was proved in \cite{bing-2}.
\begin{lem}[Section 5 in \cite{bing-2}]\label{2_podcont}
    Let $X$ be a chainable continuum. Then a point $x\in X$ is an endpoint of $X \iff $ the condition 
``if each of two subcontinua of $X$ contains $x$, then one of the subcontinua contains the other" is satisfied.
\end{lem}
Following Rosenholtz \cite{rosenholtz}, we say that a point $x\in X$ is an \textbf{absolute endpoint} of $X$ if 
for every $\varepsilon>0$ there is $\delta>0$ such that if $\{d_1,...,d_n\}$ is a $\delta$-chain covering $X$ and $x\in d_k$, then $\bigcup_{j=1}^k d_j \subseteq B(x,\varepsilon)$ or $\bigcup_{j=k}^n d_j \subseteq B(x,\varepsilon)$, where $B(x, \varepsilon)$ is an open ball around $x$ of radius $\varepsilon$.

Below we indicate the relationships between the endpoints and the absolute endpoints of chainable continua and the minimal or maximal points of chainable continua equipped with ultrafilter orders.

\begin{thm}\label{char_endpoints}
    Let $X$ be a chainable continuum and let $x\in X$. Then
    \begin{enumerate}
        \item $x$ is an endpoint of $X \iff$ there exists an ultrafilter order $\leq_{\U}^{\D}$ on $X$ such that $x$ is minimal or maximal with respect to $\leq_{\U}^{\D}$,
        \item $x$ is an absolute endpoint of $X \Longrightarrow$ for every ultrafilter order $\leq_{\U}^{\D}$ on $X$ point $x$ is minimal or maximal with respect to $\leq_{\U}^{\D}$.
    \end{enumerate}
\end{thm}

\begin{proof}[Proof of \textit{$(1)$}]
$\Longrightarrow$ Suppose that $x$ is an endpoint of a chainable continuum $X$. Then for $n\in \N$ we can take a chain $D_n$, covering $X$, such that $x\in d_{1,n}$ and $mesh(D_n)<\frac{1}{n}$. Let $\mc{D} = \{D_n\}_{n\in \N}$ be a sequence of such chains and let $\mc{U}$ be a non-principal ultrafilter on $\N$. Then we can observe that $x$ is a minimal point in the ultrafilter order $\leq_{\U}^{\D}$. 

$\Longleftarrow$ Without loss of generality we can assume that $x$ is minimal with respect to an ultrafilter order $\leq_{\mc{V}}^{\mc{E}}$, for a maximal $x$ the argument is analogous. 
Suppose that $K_1$ and $K_2$ are subcontinua of $X$, both containing a point $x$, which is minimal with respect to the order $\leq_{\mc{V}}^{\mc{E}}$. By Lemma \ref{arc_help} we know that $K_1$ and $K_2$ are nondegenerate intervals in the linearly ordered space $(X,\tau_{\mc{V}}^{\mc{E}})$. Since both $K_1$ and $K_2$ contain a minimal point $x$, we can deduce that $K_1 \subseteq K_2$ or $K_2 \subseteq K_1$. By Lemma \ref{2_podcont} this is equivalent to $x$ being an endpoint.

\textit{Proof of (2)}:
Suppose that $x$ is an absolute endpoint in $X$ and let $\leq_{\U}^{\D}$ be an ultrafilter order on $X$. By \cite[Theorem 1.0 (2)]{rosenholtz}, $X \setminus \{x\}$ is a composant of X. Since composants are intervals in the linearly ordered space $(X, \tau_{\U}^{\D})$ (this follows from Lemma \ref{arc_help}), we know that $X\setminus \{x\}$ is a subinterval of $(X,\tau_{\U}^{\D})$, hence $x$ must be maximal or minimal in $(X,\leq_{\U}^{\D})$. 
\end{proof}

We do not know if the implication in the part (2) of the above theorem can be reversed, see Question \ref{absolute_endpoints}.
Rosenholtz proved that $x$ is an absolute endpoint of a chainable continuum $X$ precisely when $X \setminus \{x\}$ is a composant of $X$ \cite[Theorem 1.0]{rosenholtz}. It is clear that there is no point $x$ in the Knaster continuum $K$ satysfying this condition, so there is no absolute endpoint in $K$. In particular, the point $(0,0)$ is an endpoint of $K$ (for more details see Subsection \ref{example}), but it is not an absolute endpoint. In the case of the Knaster continuum we have the following.

\begin{prp}\label{non-homeo}
There is an ultrafilter order 
on the Knaster continuum $K$ such that $(0,0)$ is neither minimal nor maximal point in this order.

There exist two ultrafilter orders on $K$ that yield non-homeomorphic order topologies on $K$.
\end{prp}

\begin{proof}
    Let $\mc{V}$ be a non-principal ultrafilter on $\N$ such that $\{2,4,6,8,\dots\} \in \mc{V}$. Let $X = \varprojlim (I_i,f_i)_{i=1}^{\infty}$ be an inverse limit representation of  the Knaster continuum described below Definition \ref{knasterowskie}. Recall that for every $i$, $f_i = f$, where $f$ is a map described by the Formula \ref{knaster_function}.

 \begin{figure}[h]
  \centering  \includegraphics[width=0.9\textwidth]{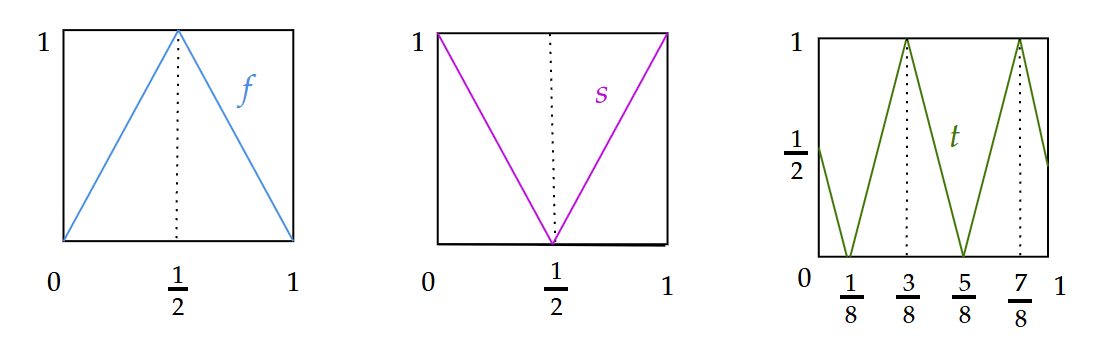}
  \caption{Functions $f$, $s$ and $t$}\label{maps}
  \end{figure}

    \begin{figure}[h]
  \centering  \includegraphics[width=0.9\textwidth]{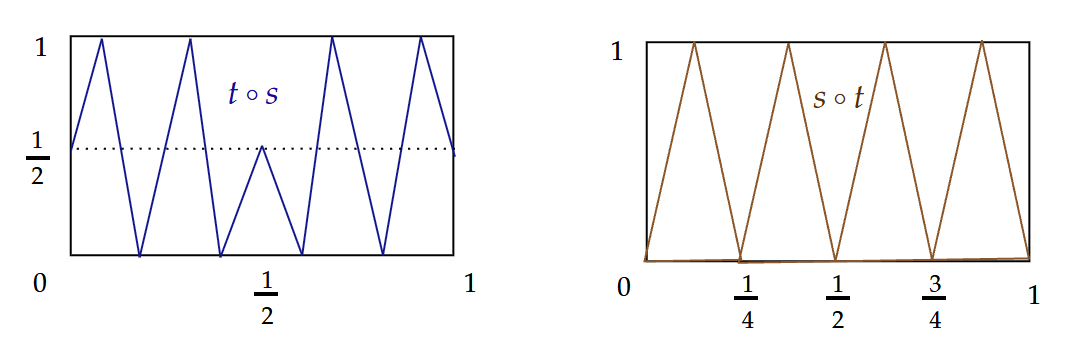}
  \caption{Functions $t\circ s$ and $s\circ t$}\label{maps2}
  \end{figure}

    Below we describe another inverse limit representation of the Knaster continuum. Let $s,t:I\to I$ be the piecewise linear maps on the compact interval, depicted on Figure \ref{maps}. For each $i \geq 1$ and for $x\in I$ we define map $g_i:I\to I$ by

    \begin{equation*}
g_i(x)=
    \begin{cases}
        s(x) & \text{if } i \text{ is odd},\\
        t(x) &  \text{if } i \text{ is even}.
    \end{cases}
\end{equation*}

    Let $Y = \varprojlim (I_i,g_i)_{i=1}^{\infty}$. One can easily see that $s\circ t = f^3$, hence $X$ is homeomorphic to $Y$. Note that the point $(0,0)$ in the Knaster continuum $K$ corresponds to the point $(0,0,0,0\dots)$ in $X$ and to the point $(0,\frac{1}{2},0,\frac{1}{2},\dots)$ in $Y$. Clearly, $\frac{1}{2}$ is a fixed point of $t\circ s$. One can observe that $t\circ s$ has more fixed points (see Figure \ref{maps2}),  in particular there are $p_1, p_2 \in [0,1]$, such that both of $p_i$ are fixed points of $t\circ s$ and $p_1<\frac{1}{2}$, $p_2>\frac{1}{2}$. Let $r_1 = s(p_1)$ and $r_2 = s(p_2)$.

    Then points $(r_1,p_1,r_1,p_1,\dots)$ and $ = (r_2,p_2,r_2,p_2,\dots)$ are in $Y$ and moreover 
    $$ (r_1,p_1,r_1,p_1,\dots) \leq_{\V}^{(I_i,g_i)_{i=1}^{\infty}} (0,\frac{1}{2},0,\frac{1}{2},\dots) \leq_{\V}^{(I_i,g_i)_{i=1}^{\infty}} (r_2,p_2,r_2,p_2,\dots).$$
This shows that the point $(0,\frac{1}{2},0,\frac{1}{2},\dots)$ is neither minimal nor maximal in $Y$ with respect to the order $\leq_{\V}^{(I_i,g_i)_{i=1}^{\infty}}$.
Let $\leq_{\mc{V}}^{\mc{E}}$ be the ultrafilter order on $K$ generated by an order $\leq_{\V}^{(I_i,g_i)_{i=1}^{\infty}}$ on $Y$ (such order exists by Theorem \ref{transfer}). Then $(0,0)$ is neither minimal nor maximal point in $(K,\leq_{\mc{V}}^{\mc{E}} )$.

Let $C_0$ be the composant of $K$ containing the point $(0,0)$. One can easily check that $(0,0)$ is a maximal point in $C_0$ with respect to the order $\leq_{\V}^{\E}$.
Since the composants of the Knaster continuum  coincide with the arc components, by Corollary \ref{non-mixing} we know that for every two different composants $C_1, C_2$ of $K$ we have $$C_1 \leq_{\V}^{\E} C_2 \text{ or } C_2 \leq_{\V}^{\E} C_1. $$

Note that by the definition, $C\leq^{\E}_{\V} C$ is not true if $C$ is a composant (since $C$ has at least two points).
    
There are two cases:
    \begin{enumerate}
        \item for every composant $C$ of $K$ such that $C_0 \leq_{\V}^{\E} C$ there exists a composant $D$ in $K$ such that $$C_0 \leq_{\V}^{\E} D  \leq_{\V}^{\E} C,$$
        \item there is a composant $C$ in $K$ such that $C_0 \leq_{\V}^{\E} C$ and there is no composant $D$ in $K$ such that $$C_0 \leq_{\V}^{\E} D  \leq_{\V}^{\E} C\, .\footnote{Actually, we can prove that the case (2) is not possible, but our proof is much longer than the above discussion of this case.}$$ 
    \end{enumerate}

     We observe that in both above cases the space $(K, \tau_{\mc{U}}^{\mc{D}})$, where $\tau_{\mc{U}}^{\mc{D}}$, is the order topology generated by an ultrafilter order on $K$ described in Subsection \ref{example}, is not homeomorphic to the space $(K, \tau_{\mc{V}}^{\mc{E}})$.
    
   \begin{itemize}
       \item In the case $(1)$ the space $(K,\tau_{\mc{V}}^{\mc{E}} )$ is not locally connected at the point $(0,0)$ and the space $(K, \tau_{\mc{U}}^{\mc{D}})$ is locally connected.
       \item In the case $(2)$  we have:
    $$(K, \tau_{\V}^{\E}) \stackrel{\text{homeo}}{\simeq} \bigoplus_{j\in J} Y_j, $$
    where all $Y_j$ are homeomorphic to the open interval $(0,1)$. By Theorem \ref{topo1} spaces $(K, \tau_{\mc{V}}^{\mc{E}})$ and $(K, \tau_{\mc{U}}^{\mc{D}})$ are not homeomorphic.  
   \end{itemize}
We conclude that $K$ admits two ultrafilter orders such that the order topologies on $K$ generated by these orders produce nonhomeomorphic spaces.
\end{proof}

\section{Questions}
We state here some open questions.

In Theorem \ref{transfer} we have proved that every limit-ultrafilter order on a chainable continuum is an ultrafilter order. We would like to ask if the converse of Theorem \ref{transfer} is true.

\begin{que}\label{two_approaches}
   Is every ultrafilter order a limit-ultrafilter order?
\end{que}

One can easily show that limit-ultrafilter orders are dense. We conjecture that ultrafilter orders defined using sequence of chains are also dense orders.   
\begin{que}
    Is every ultrafilter order a dense order?
\end{que}
We also ask the following stronger question.
\begin{que}
    Is it true that for every chainable continuum $X$, every ultrafilter order $\leq_{\U}^{\D}$ on $X$ and any two distinct points $a,b \in X$ there exists a non-degenerate subcontinuum $K\subseteq X$ such that for all $x\in K$ we have: $a \leq_{\U}^{\D} x \leq_{\U}^{\D} b$?  
\end{que}

Again, one can show that the above question has an affirmative answer  for limit-ultrafilter orders.

\begin{que}\label{1ap}
    Does the order topology $\tau_{\U}^{\D}$ generated by any ultrafilter order $\leq_{\U}^{\D}$ on any chainable continuum $X$ have a countable character in every $x\in (X,\tau_{\U}^{\D})$?
\end{que}

By Corollary \ref{properties} 
there is an ultrafilter order topology on the Knaster continuum that is non-connected and non-compact. 
Moreover, by the proof of Proposition \ref{non-homeo}, there is another topology on the Knaster continuum with the same properties.
We ask if any order topology on any indecomposable chainable continuum also has those properties.
\begin{que}
     Is it true that order topology generated by any ultrafilter order on any indecomposable chainable continuum is non-connected and non-compact?
\end{que}

\begin{que}\label{absolute_endpoints} 
	Let $X$ be a chainable continuum and let $x$ be a point of $X$ which is not an absolute endpoint. 
    Does there exist an ultrafilter order on $X$ such that $x$ is neither minimal nor maximal with respect to this order?
\end{que}

\begin{prob}
    Describe order topologies generated by ultrafilter orders on the pseudoarc.
\end{prob}

\subsection*{Acknowledgments} As mentioned in the introduction, the concept of ultrafilter orders on chainable continua was created by Jakub Różycki.
We would like to thank him for sharing his fruitful idea with us and allowing us to explore this notion.

\end{document}